\documentclass[11pt,a4paper]{amsart}
\usepackage{amssymb,amsmath,amsthm}
\usepackage{tikz}
\usepackage{tikz-cd}
\usepackage{enumitem}
\usepackage{amsaddr}
\usepackage{etoolbox}
\usepackage{upgreek}\usetikzlibrary{matrix}
\usepackage[utf8]{inputenc}
\usepackage{hyperref}
\usepackage{thmtools,thm-restate}
\usepackage{float}
\usepackage{cleveref}
\usetikzlibrary{decorations.pathmorphing}
\usepackage{subcaption}
\usepackage{mathrsfs}

\usetikzlibrary{fit, shapes.geometric}
\usetikzlibrary{backgrounds}

\newcommand{\C}{\mathcal}
\newcommand{\s}{\mathsf}
\newcommand{\f}{\mathfrak}
\newcommand{\R}{\mathrm}
\newcommand{\B}{\mathbf}
\newcommand{\ang}[1]{\langle#1\rangle}
\newcommand{\Tfac}{T^{\mathrm{fac}}}
\newcommand{\Tim}{T^{\mathrm{im}}}
\newcommand{\Hom}{\R{Hom}_{\Lambda}(M_1,M_2)}

\newcounter{thmcount}
\newtheorem{defn}{Definition}[section]

\newtheorem{lem}[defn]{Lemma}

\newtheorem{prop}[defn]{Proposition}
\newtheorem*{prop*}{Proposition}
\newtheorem{thm}[defn]{Theorem}
\newtheorem{conj}[defn]{Conjecture}
\newtheorem{cor}[defn]{Corollary}

\newtheorem*{claim*}{Claim}

\theoremstyle{remark}

\theoremstyle{remark}

\theoremstyle{remark}
\newtheorem{exmp}[defn]{Example}
\AtEndEnvironment{exmp}{\hfill$\Diamond$}
\theoremstyle{remark}

\theoremstyle{remark}

\theoremstyle{remark}

\theoremstyle{remark}

\theoremstyle{remark}
\newtheorem{que}[defn]{Question}
\theoremstyle{remark}
\newtheorem{rmk}[defn]{Remark}
\theoremstyle{remark}

\numberwithin{equation}{section}

\title{Generalised tree modules: Hom-sets and indecomposability}
\author{Annoy Sengupta and Amit Kuber\\Corresponding author: Amit Kuber}

\address{Department of Mathematics and Statistics\\Indian Institute of Technology, Kanpur\\ Uttar Pradesh, India}
\email{annoysgp20@iitk.ac.in, askuber@iitk.ac.in}

\date{}

\keywords{zero-relation algebra, tree module, graph map, indecomposability, Dynkin quiver of type D}
\subjclass[2020]{16G20}

\begin{document}

\begin{abstract}
For a zero-relation algebra over a field $\mathcal K$, Crawley-Boevey introduced the concept of a tree module and provided a combinatorial description of a basis for the space of homomorphisms between two tree modules--the basis elements are called graph maps. The indecomposability of tree modules is essentially due to Gabriel. We relax a condition in the definition of a tree module to define \emph{generalised tree modules} and when $\mathrm{char}(\mathcal K)\neq2$, under a certain condition, provide a combinatorial description of a finite generating set for the space of homomorphisms between two such modules--we call the generators \emph{generalised graph maps}. As an application, we provide a sufficient condition for the (in)decomposability of certain generalised tree modules. We also show that all indecomposable modules over a Dynkin quiver of type $\mathbf D$ are isomorphic to generalised tree modules--this result also follows from a theorem of Ringel which states that all exceptional modules over the path algebra $\mathcal KQ$ of a finite quiver $Q$ are generalised tree modules.
\end{abstract}

\maketitle

\section{Introduction}\label{sec: intro}
Let $\C K$ be a field with $\R{char}(\C K)\neq2$. Let $ Q:=( Q_0, Q_1,\varsigma,\varepsilon)$ be a locally finite quiver with vertex set $Q_0$, arrow set $Q_1$, and source and target functions $\varsigma,\varepsilon:Q_1\to Q_0$ respectively. Further let $\rho$ be a set of paths in $ Q$ of length at least $2$ such that for each $v\in Q_0$, there is some $n\geq0$ such that $\ang\rho$ contains all paths of length exceeding $n$ with source or target at $v$. The quotient algebra $\Lambda:=\C KQ/\ang{\rho}$ associated with the locally bound quiver $(Q,\rho)$ is called a \emph{zero-relation algebra}. There is an equivalence between the category $\Lambda\text{-}\R{mod}$ of finite-dimensional left $\Lambda$-modules and the category of finite-dimensional bound $\C K$-representations of $(Q,\rho)$. In other words, each left $\Lambda$-module can be viewed as a collection $((V_j)_{j\in  Q_0},(\varphi_{\gamma})_{\gamma\in  Q_1})$, where $V_j$ is a $\C K$-vector space for each $j\in  Q_0$ and $\varphi_\gamma:V_{\varsigma(\gamma)}\to V_{\varepsilon(\gamma)}$ is a linear map for each $\gamma\in  Q_1$.

The proofs of indecomposability of certain modules over such algebras could be traced back to the universal covering technique introduced by Gabriel \cite{gabriel1981covering}, and Bongartz and Gabriel \cite{gabrielbongartz}. However, a complete classification of finitely generated modules over a zero-relation algebra is an open problem.

The main objects of interest in this paper are generalisations of certain modules known as \emph{tree modules}.
\begin{defn}\cite[\S~1]{crawley1989maps}\label{defn: tree module}
A \emph{tree} is a finite quiver $T=(T_0,T_1,s,t)$ whose underlying undirected graph is simply connected. We assume $T_0\subset\mathbb N$, where $\mathbb N$ is the set of non-negative integers. We denote by $V_T$ the $\C KT$-module corresponding to the $\C K$-representation $((\C K)_{n\in T_0},(1_{\C K})_{a\in T_1})$. Let $F:T\to Q$ be a quiver morphism, i.e., functions $F_j:T_j\to Q_j$ for $j=1,2$ compatible with source and target functions, satisfying
\begin{enumerate}
    \item \label{bound quiver morphism 1} (Bound quiver morphism) there is no path $p$ in $T$ such that $F(p)\in\ang\rho$; and
    \item \label{tree module condition} (Tree module condition) $F(a)\neq F(b)$ for distinct $a,b\in T_1$ with $s(a)=s(b)$ or $t(a)=t(b)$.
\end{enumerate}
A \emph{tree module} is a module of the form $F_\lambda(V_T)$, where the \emph{push-down functor} $F_\lambda:\C KT\text{-}\R{mod}\to\Lambda\text{-}\R{mod}$ is defined by
$$(F_\lambda((U_n)_{n\in T_0},(\psi_a)_{a\in T_1}))_j:=\bigoplus_{n\in F^{-1}(j)}U_n,\quad (F_\lambda((U_n)_{n\in T_0},(\psi_a)_{a\in T_1}))_\gamma:=\bigoplus_{a\in F^{-1}(\gamma)}\psi_a.$$
\end{defn}

A quiver morphism $F:T\to Q$ satisfying Condition \ref{bound quiver morphism 1} will be written as $F:T\to( Q, \rho)$ to emphasise the set $\rho$ of relations. 

The next result is essentially due to Gabriel and uses universal covering techniques.

\begin{thm}\cite[\S~3.5,4.1]{gabriel1981covering}
Tree modules over a zero-relation algebra are indecomposable.
\end{thm}

A combinatorial description of a basis for the space $\R{Hom}_\Lambda(M_1,M_2)$ of homomorphisms between two tree modules $M_1$ and $M_2$ was given by Crawley-Boevey in \cite{crawley1989maps}.

\begin{defn}
A subtree $\Tfac$ of a tree $T$ is said to be a \emph{factor subtree} if for each arrow $n\xrightarrow{a}m$ in $T_1$, whenever $m\in \Tfac_0$ then $n\in \Tfac_0$. Dually, a subtree $\Tim$ of $T$ is said to be an \emph{image subtree} if for each arrow $n\xrightarrow{a}m$ in $T_1$, whenever $n\in \Tim_0$ then $m\in \Tim_0$.
\end{defn}

\begin{thm}\cite[\S~2]{crawley1989maps}\label{thm: c-b}
For $j=1,2$, let $T^j$ be a tree, $F_j:T^j\to( Q,\rho)$ be a bound quiver morphism, and $M_j:={F_{j\lambda}}(V_{T^1})$ be the associated tree module. Then the Hom-set $\Hom$ has as a basis the set of triples $(\Tfac,\Tim,\Phi)$, known as \emph{graph maps}, satisfying the following conditions:
\begin{enumerate}
    \item $\Tfac$ is a non-empty factor subtree of $T^1$;
    \item $\Tim$ is an image subtree of $T^2$; and
    \item $\Phi:\Tfac\to \Tim$ is a quiver isomorphism satisfying ${F_2}\circ\Phi={F_1}$.
\end{enumerate}
\end{thm}

\begin{rmk}\label{rmk: Flambda functor}
In the absence of Condition \ref{tree module condition} in \Cref{defn: tree module} for a bound quiver morphism $F:T\to( Q,\rho)$, the following assignments still define a functor (by an abuse of notation) $F_\lambda:\C KT\text{-}\R{mod}\to\Lambda\text{-}\R{mod}$:
$$(F_\lambda((U_n)_{n\in T_0},(\psi_a)_{a\in T_1}))_j:=\bigoplus_{n\in F^{-1}(j)}U_n,\quad(F_\lambda((U_n)_{n\in T_0},(\psi_a)_{a\in T_1}))_\gamma:=\sum_{a\in F^{-1}(\gamma)}\psi_a.$$  
\end{rmk}
Now we are ready to define the main object of study in this paper.
\begin{defn}\label{defn: gen tree module}
A \emph{generalised tree module} is a module of the form $F_{\lambda}(V_T)$, where $T$ is a tree, $F:T\to ( Q,\rho)$ is a bound quiver morphism, and $F_\lambda$ is the functor as defined in \Cref{rmk: Flambda functor}.
\end{defn}

If $Q$ is finite and $\rho=\emptyset$, the notion of a generalised tree module coincides with the notion of a \emph{tree module}\footnote{Caveat: Ringel's tree modules are different from Crawley-Boevey's tree modules.} introduced by Ringel \cite{ringelexceptional}. Suppose $M\in\C KQ\text{-}\R{mod}$ admits a basis $\C B$ for its underlying vector space. The \emph{coefficient quiver} $\Gamma(M,\C B):=(\C B,\Gamma_1)$ of $M$ with respect to the basis $\C B$ is defined as follows: given $b,b'\in\C B$ and $\alpha\in Q_1$, there is an arrow $(\alpha,b,b')$ from $b$ to $b'$ in $\Gamma_1$ if $\alpha\cdot b$ admits a non-zero coefficient of $b'$.  If the coefficient quiver $\Gamma(M,\C B)$ is a tree, then there are exactly $|\C B|-1$ arrows of $\Gamma$. Then it follows from \cite[Property~2]{ringelexceptional} that whenever $(\alpha,b,b')\in\Gamma_1$ then the coefficient of $b'$ in the expression of $\alpha\cdot b$ could be chosen to be equal to $1$.

Given a generalised tree module $M:=F_\lambda(V_T)$, there is a canonical $\C K$-basis $\C B:=\{v_n\mid n\in T_0\}$ of $M$ induced by $F$ that admits a natural isomorphism $\Gamma(M,\C B)\cong T$.

A generalised tree module is not necessarily indecomposable as \Cref{exmp: gtm can be decomposable} demonstrates. Some special generalised tree modules called \emph{rooted tree modules} over rooted tree quivers were studied by Kinser \cite{Kinser}. Subsequently, Katter and Mahrt provided a necessary and sufficient condition \cite[Lemma~2]{katter_mahrt} for the indecomposability of some rooted tree modules.

\begin{exmp}\label{exmp: gtm can be decomposable}
If $\B A_2$ is the quiver $\B1\xrightarrow{\B a}\B 2$, then the $\C K\B A_2$-module $\C K\oplus\C K\xrightarrow{(1\ 1)}\C K$ is a decomposable generalised tree module associated with the data of the tree $T$ given by  $1\xrightarrow{a}2\xleftarrow{b}3$ with quiver morphism $F:T\to \B A_2$ given by $F(1)=F(3)=\B 1, F(2)=\B 2, F(a)=F(b)=\B a$.
\end{exmp}

We use the following characterisation of the indecomposability of $\Lambda$-modules.
\begin{prop}\cite[Lemma~I.4.6, Corollary~I.4.8]{assem2006elements}\label{prop: indecomposability condn assem}
If $M$ is a finite-dimensional $\Lambda$-module, then $M$ is indecomposable if and only if the endomorphism algebra $\R{End}_\Lambda(M)$ contains only two idempotents, viz. $\B 0$ and $\B 1_M$.
\end{prop}

In view of the above characterisation, in an attempt to provide some necessary and sufficient conditions for the indecomposability of generalised tree modules, we first generalise \Cref{thm: c-b} that describes a basis of the Hom-set between two tree modules to a theorem (\Cref{thm: main 1}) that provides a combinatorial description of a finite generating set for the Hom-set between two generalised tree modules under certain conditions. To state this theorem clearly, we need to set up some notations and introduce some terminology.

Suppose that $T^1$ and $T^2$ are trees, $F_1:T^1\to  Q$, $F_2:T^2\to  Q$ are bound quiver morphisms, and $M_1:={F_{1\lambda}}(V_{T^1})$ and $M_2:={F_{2\lambda}}(V_{T^2})$ are generalised tree modules. Denote the canonical $\C K$-basis of $M_1$ by $\{v_n\}_{n\in T^1_0}$ and that of $M_2$ by $\{w_m\}_{m\in T^2_0}$. Let $\C P:=(\C P_0,\C P_1)$ denote the pullback quiver of the morphism $F_1:T^1\to Q$ along $F_2:T^2\to Q$ in the category of quivers and quiver morphisms.

If $M_1$ and $M_2$ are tree modules then each graph map $(\Tfac,\Tim,\Phi)$ can be viewed as a subquiver $\C G$ of $\C P$ induced by the set $\C G_0:=\{(n,\Phi(n))\mid n\in\Tfac_0\}$; the corresponding homomorphism is $\C {H_G}(v_n):=\sum_{(n,m)\in\C G_0}w_m$.
Note that the coefficients of $w_m$ in the expression of $\C {H_G}(v_n)$ are either $0$ or $1$. However, if either $M_1$ or $M_2$ is not a tree module, then as a result of dropping Condition \ref{tree module condition} of \Cref{defn: tree module}, we need to allow yet another coefficient, namely $-1$, while ensuring finiteness of a generating set for $\Hom$--the next example demonstrates the need for this adjustment.

\begin{exmp}\label{exmp: motivation of flip}
Consider the quiver $\B A_2$, and trees $T^1$ and $T^2$ defined in Figure \ref{fig:motivation of flip} with $F_1(1)=\B{1}$, $F_1(2)=F_1(3)=F_2(4)=\B{2}$ and $F_1(a)=F_1(b)=\B{a}$.

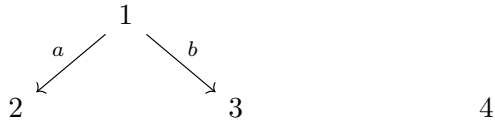
\begin{figure}[H]
    \[\begin{tikzcd}
  & 1 \arrow[ld, "a"'] \arrow[rd, "b"] &   &   \\
2 &                                    & 3 &&& 4
\end{tikzcd}\]
    \caption{Two trees $T^1$ (in the left) and $T^2$ (in the right)}
    \label{fig:motivation of flip}
\end{figure}

The assignment $v_1\mapsto0,\ v_2\mapsto w_4,\ v_3\mapsto-w_4$ describes a ``sign-flip" homomorphism as the only element in a basis of $\R{Hom}_{\C K\B A_2}(M_1,M_2)$.
\end{exmp}

Our goal is to find a finite generating set for $\Hom$ by choosing generators as certain elements of $\C P_0\times\{-1,0,1\}$. Examples like the above motivate us to record sign-flips as (undirected) edges between some vertices of $\C P$--we call the resulting structure the \emph{pullback network} $\C N[1]$ associated with the pair $(M_1,M_2)$. Moreover, to explicitly record the sign, we also define the \emph{$2$-covering network} $\C N[2]$ whose vertices are $(n,m,j)$, where $(n,m)\in\C P_0$ and $j\in\{-1,1\}$, equipped with a natural surjection $\pi:\C N[2]\to\C N[1]$. The network $\C N[2]$ is indeed a $2$-cover of $\C N[1]$ for a fiber of a vertex (resp., an arrow or an edge) in $\C N[1]$ with respect $\pi$ contains exactly two vertices (resp., two arrows or two edges).

The set $\C M_0$ of vertices of an arbitrary subnetwork $\C M$ of $\C N[2]$ can be used to define a $\C K$-linear map $\C{H_M}:M_1\to M_2$ by $\C{H_M}(v_n):=\sum_{(n,m,j)\in\C M_0}jw_m$. However, $\C{H_M}$ may not be a homomorphism for it may fail to satisfy $\C{H_M}(\alpha\cdot v_n)=\alpha\cdot(\C{H_M}(v_n))$ for some $\alpha\in Q_1$ and $n\in T^1_0$. When $M_1$ and $M_2$ are tree modules, such identities are satisfied because of the fact that $\Tfac$ and $\Tim$ are factor and image subtrees of $T^1$ and $T^2$ respectively. For the case of generalised tree modules, we ensure that these identities are satisfied by requiring $\C M$ to satisfy two key conditions namely, \emph{completeness}(=existence) (\Cref{defn: complete subnetwork}) and \emph{$\C R[2]$-freeness}(=uniqueness) (\Cref{defn: Rj free}), where $\C R[2]$ documents the data of certain ``blocked walks'' in the network $\C N[2]$. We use these two conditions to declare some connected subnetworks of $\C N[2]$ \emph{generalised graph maps} (\Cref{defn: ggm}) from $M_1$ to $M_2$. The set $\{\C{H_G}\mid\C G\text{ is a generalised graph map from $M_1$ to $M_2$}\}$ of homomorphisms is not necessarily linearly independent (\Cref{exmp: ggm}) but it ``should'' provide a finite combinatorially-described generating set for $\Hom$. Unfortunately, we encounter some obstacles along the way. 

Note that if $(n,m,1)$ and $(n,m,-1)$ are both vertices of a subnetwork $\C M$ of $\C N[2]$, then the coefficient of $w_m$ in the expression of $\C{H_M}(v_n)$ is $0$. We call a non-empty connected complete $\C R[2]$-free subnetwork $\C M$ of $\C N[2]$ a \emph{ghost} (\Cref{defn: ghost}) if $\C{H_M}=\B0$--our choice of terminology comes from the fact that these combinatorial objects are algebraically invisible. \Cref{exmp: ghost} shows a concrete example of a ghost. Say that the pair $(M_1,M_2)$ of generalised tree modules is \emph{ghost-free} if the associated $2$-covering network $\C N[2]$ does not contain a ghost.

Now we are ready to state the main results of the paper.
\begin{restatable}{theom}{mainone}\label{thm: main 1}
Suppose $(M_1,M_2)$ is a ghost-free pair of generalised tree modules. Given a homomorphism $\C H:M_1\to M_2$, there exist $\mu_1,\cdots,\mu_N\in\C K$ and generalised graph maps $\C G_1,\cdots,\C G_N$ from $M_1$ to $M_2$ such that $$\C H=\sum_{l=1}^N\mu_l\C H_{\C G_l}.$$
\end{restatable}

We ask in \Cref{ques: main1} whether the ghost-free hypothesis could be dropped from the above theorem. Using the above theorem, we give a checkable sufficient condition for the indecomposability of some generalised tree modules.

\begin{restatable}{theom}{maintwo}\label{thm: main 2}
Suppose $M:=F_\lambda(V_T)$ is a generalised tree module such that the pair $(M,M)$ is ghost-free. Suppose the following conditions hold:
\begin{enumerate}
    \item for every subtree of $T$ of the form $n_1\xleftarrow{a}n\xrightarrow{b}n_2$ or $n_1\xrightarrow{a}n\xleftarrow{b}n_2$ with $n_1\neq n_2$ and satisfying $F(a)=F(b)$, and for every generalised graph map $\C G$ from $M$ to $M$, we have that $(n_1,n_2,j)$ is not a vertex of $\C G$ for any $j\in\{-1,1\}$; and

    \item for $n_1\neq n_2$ in $T_0$, if there exists a generalised graph map $\C G$ such that $(n_1,n_2,j)$ is a vertex of $\C G$ for some $j\in\{-1,1\}$, then there does not exist a generalised graph map $\C G'$ containing the vertex $(n_2,n_1,j')$ for some $j'\in\{-1,1\}$.
\end{enumerate}
Then $M$ is indecomposable.
\end{restatable}

For a subclass of the class of generalised tree modules, we provide a sufficient condition (\Cref{thm: conv main2}) for decomposability, and conjecture (\Cref{conj}) that any generalised tree module satisfying this condition is decomposable. An affirmative answer to \Cref{ques: main1} as well as \Cref{conj} would provide a necessary and sufficient condition for the indecomposability of an arbitrary generalised tree module.

For a quiver $Q$ and an indecomposable $M\in\C KQ\text{-}\R{mod}$, Ringel \cite{ringelexceptional} showed that if $M$ is exceptional, i.e. $\R{Ext}^1_{\C KQ}(M,M)=0$, then $M$ is isomorphic to a generalised tree module. It is known  that all indecomposable modules over a Dynkin quiver are exceptional, and hence are generalised tree modules by the above theorem. For a dimension vector $\overline{d}$ of an indecomposable module over a Dynkin quiver of type $\B D$, we give in \S~\ref{sec:Dn indecomposable} an explicit construction of a generalised tree module $M(\overline{d})$ with $\overline{\dim}(M(\overline{d}))=\overline{d}$ and, as an application of \Cref{thm: main 2}, we show that $M(\overline{d})$ is indecomposable.

The paper is organised as follows. In \S~\ref{sec: cov network}, we introduce the concept of a network, define networks $\C N[1]$ and $\C N[2]$ associated with a pair of generalised tree modules, and discuss their properties. Generalised graph maps are introduced in \S~\ref{sec: ggm}. The proof of \Cref{thm: main 1} is complete by the end of \S~\ref{sec: main}, and it relies on a key lemma that occupies most of \S~\ref{sec: R2 free from complete}. \Cref{thm: main 2} along with a result in the direction of its converse are proved in \S~\ref{sec: indecomposability}. Applications of the results from \S~\ref{sec: indecomposability} to modules over Dynkin quivers of type $\B D$ are given in \S~\ref{sec:Dn indecomposable}.

\section{Networks associated with a pair of generalised tree modules}\label{sec: cov network}
Throughout the rest of the paper, we assume that $T^1$ and $T^2$ are trees, $F_1:T^1\to(Q,\rho)$, $F_2:T^2\to(Q,\rho)$ are bound quiver morphisms, and $M_1:={F_{1\lambda}}(V_{T^1})$ and $M_2:={F_{2\lambda}}(V_{T^2})$ are generalised tree modules with natural $\C K$-bases $\{v_n\}_{n\in T^1_0}$ and $\{w_m\}_{m\in T^2_0}$ respectively.

In this section, we define some structures that will help us to generalise the definition of a graph map between two tree modules. The relaxation of Condition \ref{tree module condition} in the definition of a tree module forces the existence of some homomorphisms which are different from the ones appearing in the basis of the Hom-set between two tree modules as described in \Cref{thm: c-b}. Such a sign-flip homomorphism was shown in \Cref{exmp: motivation of flip}. We want to look at the assignments of the basis elements in a homomorphism, as in that example, as pairs of basis elements, which will be vertices of some ``network'' associated with a pair of generalised tree modules; we introduce this concept now.

\begin{defn}\label{defn: network}
A \emph{network} $\C N$ is defined as the pentuple $(\C N_0,\C N_1,\sigma,\tau,\C E)$, where $(\C N_0,\C N_1,\sigma,\tau)$ is a quiver and $(\C N_0,\C E)$ is a simple undirected graph.    
\end{defn}

The same concept could be found under the name of \emph{mixed graph} introduced by Harary and Palmer (see \cite{harary1966enumeration}). A subnetwork $\C M$ of $\C N$ is \emph{connected} if the underlying graph of $\C M$ is connected. Given a set of vertices $\C V\subseteq\C N_0$, denote by $\ang{\C V}$ the induced subnetwork of $\C N$ with the set of vertices $\C V$.

Borrowing motivation from strings in a string algebra (see \cite[\S~2]{sengupta2024characterisation}), we define \emph{traversals} in a network $\C N$. For each $\alpha\in\C N_1$, we introduce a formal inverse $\alpha^{-1}$ and extend the functions $\sigma$ and $\tau$ so that $\sigma(\alpha^{-1}):=\tau(\alpha)$ and $\tau(\alpha^{-1}):=\sigma(\alpha)$. Let $\C N_1^{-1}:=\{\alpha^{-1}\mid\alpha\in\C N_1\}$. We will refer to the elements of the set $\C N_1\sqcup\C N_1^{-1}\sqcup\C E$ as \emph{links}.

\begin{defn}
A sequence of links $\f t:=\alpha_n\cdots\alpha_1$ is a \emph{traversal} (of length $n\geq1$) if the following conditions hold:
\begin{itemize}
    \item for all $1\leq i<n$, if $\alpha_i,\alpha_{i+1}\in\C N_1\sqcup\C N_1^{-1}$, then $\tau(\alpha_i)=\sigma(\alpha_{i+1})$;

    \item for all $1\leq i<n$, if $\alpha_i\in\C N_1\sqcup\C N_1^{-1}$ and $\alpha_{i+1}\in\C E$, then $\tau(\alpha_i)\in\alpha_{i+1}$;

    \item for all $1\leq i<n$, if $\alpha_{i+1}\in\C N_1\sqcup\C N_1^{-1}$ and $\alpha_i\in\C E$, then $\sigma(\alpha_{i+1})\in\alpha_{i}$;

    \item for all $1\leq i<n$, if $\alpha_i,\alpha_{i+1}\in\C E$, then $|\alpha_i\cap\alpha_{i+1}|=1$;

    \item for all $1\leq i<n-1$, if $\alpha_i,\alpha_{i+1},\alpha_{i+2}\in\C E$, then $\alpha_i\cap\alpha_{i+1}\cap\alpha_{i+2}=\emptyset$;

    \item for all $1\leq i<n-1$, if $\alpha_i\in \C N_1\sqcup\C N_1^{-1}$ and $\alpha_{i+1},\alpha_{i+2}\in\C E$, then $\tau(\alpha_i)\notin\alpha_{i+1}\cap\alpha_{i+2}$;

    \item for all $1\leq i<n-1$, if $\alpha_i,\alpha_{i+1}\in\C E$ and $\alpha_{i+2}\in \C N_1\sqcup\C N_1^{-1}$, then $\sigma(\alpha_{i+2})\notin\alpha_i\cap\alpha_{i+1}$;

    \item for all $1\leq i<n$, if $\alpha_i,\alpha_{i+1}\in\C N_1\sqcup\C N_1^{-1}$, then $\alpha_{i+1}\neq\alpha_i^{-1}$.
\end{itemize}
In addition to the above, we also refer to the vertices in $\C N_0$ as \emph{traversals} (of length $0$). Define $\alpha^{-1}:=\alpha$ for $\alpha\in\C E\cup\C N_0$. We will denote the length of a traversal $\f t$ by $|\f t|$. Also, we will denote the set of traversals by $\f T(\C N)$.
\end{defn}

\begin{rmk}\label{rmk: inv of a traversal}
It is simple to observe that if $\f t:=\alpha_n\cdots\alpha_1\in\f T(\C N)$, then $\f t^{-1}:=\alpha_1^{-1}\cdots\alpha_n^{-1}\in\f T(\C N)$.
\end{rmk}

Suppose $\f t\in\f T(\C N)$ and assume $\f{t}=\alpha_n\cdots\alpha_1$ if $|\f t|=:n>0$. We associate sets $\f{s}(\f{t})$ and $\f{e}(\f{t})$ of starting and ending vertices respectively defined as follows:
\begin{equation*}
    \f{e}(\f{t}):=
    \begin{cases}
        \{\f t\} &\text{ if } |\f t|=0,\\
        
        \f{t}&\text{ if }\f{t}\in\C E,\\
        
        \{\tau(\alpha_n)\}&\text{ if }\alpha_n\in\C N_1\sqcup\C N_1^{-1},\\
        
        \alpha_n\setminus\alpha_{n-1}&\text{ if }\alpha_{n-1},\alpha_n\in\C E,\\
        
        \alpha_n\setminus\{\tau(\alpha_{n-1})\}&\text{ if }\alpha_{n-1}\notin\C E,\alpha_n\in\C E,        
    \end{cases}
\end{equation*}
and $\f{s}(\f{t}):=\f{e}(\f{t}^{-1})$. We say that $\f{t}$ is a traversal from $(v,i)$ to $(v',j)$ if $(v,i)\in\f{s}(\f{t})$, $(v',j)\in\f{e}(\f{t})$, and $(v,i)\neq(v',j)$ if $\f{t}\in\C E$. It is convenient to exhibit a traversal with the help of a diagram of links instead of the sequence, where the notation $``--"$ will be used to denote a link. For example, a traversal $\alpha_3\alpha_2\alpha_1$ where $\alpha_1\in\C N_1$, $\alpha_2\in\C E$ and $\alpha_3\in\C N_1^{-1}$ can be expressed as $\tau(\alpha_1)\xrightarrow{\alpha_1^{-1}}\sigma(\alpha_1)\stackrel{\alpha_2}{\text{ ----- }}\tau(\alpha_3)\xleftarrow{\alpha_3}\sigma(\alpha_3).$

\begin{figure}[H]
    \[\begin{tikzcd}
1 \arrow[r, no head] \arrow[d, "a"'] & 2 \arrow[d, "b"] \\
3 \arrow[r, no head]                 & 4               
\end{tikzcd}\]
    \caption{A network $\C N$ used in \Cref{exmp: traversal network}}
    \label{fig:traversal network}
\end{figure}

\begin{exmp}\label{exmp: traversal network}
Consider the network $\C N$ in \Cref{fig:traversal network}. The  diagram $1\xrightarrow{a}3\text{ --- }4\xleftarrow{b}2\text{ --- }1$ shows a traversal from the vertex $1$ to itself in $\C N$.
\end{exmp}

Now we describe the construction of the \emph{pullback network} associated with the pair $(M_1,M_2)$ of generalised tree modules. In the category of quivers, consider the quiver $\C P$ which is the pullback as shown in \Cref{fig: pullback}.
\begin{figure}[H]
    \[\begin{tikzcd}
\C P \arrow[r, "\pi_1"] \arrow[d, "\pi_2"'] \arrow[dr, phantom, "\scalebox{1.5}{$\lrcorner$}", very near start] & T^1 \arrow[d, "F_1"] \\
T^2 \arrow[r, "F_2"']                                &  Q                
\end{tikzcd}\]
    \caption{Pullback diagram used in the definition of pullback network}
    \label{fig: pullback}
\end{figure}

Explicitly, the quiver $\C P:=(\C P_0,\C P_1,s,t)$ is defined by
\begin{align*}
\C P_0&:=\{(n,m)\in T^1_0\times T^2_0\mid F_1(n)=F_2(m)\};\text{ and}\\
\C P_1&:=\{(n,m)\xrightarrow{(a,b)}(n',m')\mid(n\xrightarrow{a}n')\in T^1_1,(m\xrightarrow{b}m')\in T^2_1, F_1(a)=F_2(b)\}.
\end{align*}
There is a quiver morphism $F:\C P\to Q$ defined by $F((x,y)):=F_1(x)=F_2(y)$ for $(x,y)\in\C P_0\cup\C P_1$. Recall from Example \ref{exmp: motivation of flip} that there is a basis homomorphism containing the assignments $v_2\mapsto w_4$ and $v_3\mapsto-w_4$. We wish to ``connect'' the vertices $(2,4)$ and $(3,4)$ via an (undirected) edge to encode this ``sign-flip'' homomorphism. Therefore, we define the following set of edges:
\begin{align*}
    \C E^1&:=\{\{(n,m),(n',m)\}\mid\text{ there are }(n\xleftarrow{a}n''),(n''\xrightarrow{b}n')\text{ in }T^1_1\text{ with }n\neq n',F_1(a)=F_1(b)\}\\&\cup
    \{\{(n,m),(n,m')\}\mid\text{ there are }(m\xrightarrow{c}m''),(m''\xleftarrow{d}m')\text{ in }T^2_1\text{ with }m\neq m',F_2(c)=F_2(d)\}.
\end{align*}
The network $(\C P,\C E^1)$ will be called the \emph{pullback network} associated with the pair $(M_1,M_2)$ of generalised tree modules. We will use the notation $\C N[1]:=(\C N[1]_0,\C N[1]_1,s^1,t^1,\C E^1)$ to denote this network--the choice of notation will be clear by the end of this section. \Cref{fig:situation edge N1} shows both the situations mentioned in the definition of $\C E^1$.
\begin{figure}[H]
    \begin{subfigure}{0.7\textwidth} 
        \[\begin{tikzcd}
   & n'' \arrow[ld, "a"'] \arrow[rd, "b"] &     &  &                               &  &            \\
n &                                    & n' & m               & {(n,m)} \arrow[rr, no head] &                                                      & {(n',m)}
\end{tikzcd}\]
        \caption{$F_1(n)=F_1(n')=F_2(m)$}
        \label{fig:situation edge N1 1}
    \end{subfigure}
    \\
    \begin{subfigure}{0.7\textwidth}
        \[\begin{tikzcd}
n & m \arrow[rd, "c"'] &   & m' \arrow[ld, "d"] & {(n,m)}  \arrow[rr, no head] &  & {(n,m')} \\
  &  & m'' &  &  & &      
\end{tikzcd}\]
    \caption{$F_1(n)=F_2(m)=F_2(m')$}
        \label{fig:situation edge N1 2}
    \end{subfigure}

    \caption{Two situations showing the existence of an edge in the pullback network, where portions of $T^1$, $T^2$ and $\C N[1]$ are being shown respectively from left to right.}
    \label{fig:situation edge N1}
\end{figure}
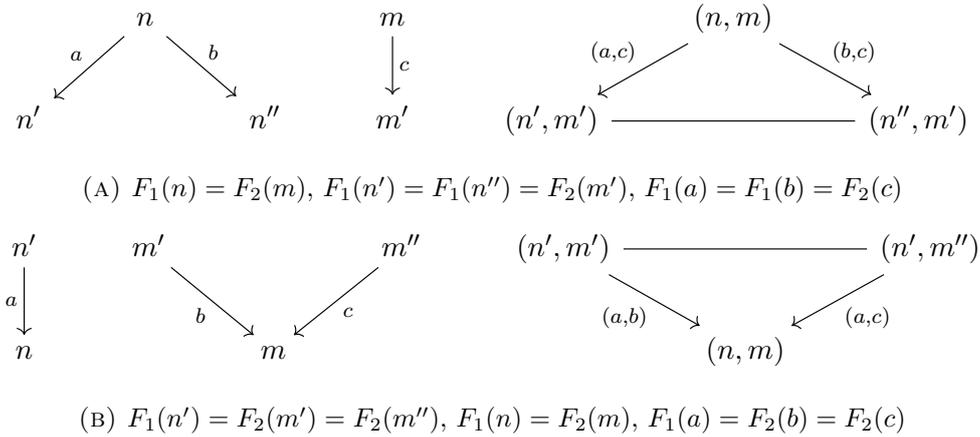

We mention below some observations about $\C N[1]$.

\begin{rmk}\label{rmk: atmost one arrow or edge between two vertices}
Since $T^1$ and $T^2$ are trees, there is at most one link from one vertex to another in $\C N[1]$, and there does not exist a directed cycle of arrows in $\C N[1]$.
\end{rmk}

To state the next remark, we need some preliminaries on alphabets and words. One may look at \cite[Chapter 1.0]{lothaire1997combinatorics} for the definition of an alphabet and a word. Let $\s A:= Q_1\cup Q_1^{-1}$ be an alphabet. Denote the empty $\s A$-word by $e$. An $\s A$-word $\s w$ is said to be \emph{reduced} if there is no subword $\alpha\alpha^{-1}$ or $\alpha^{-1}\alpha$ of $\s w$. Given an $\s A$-word $\s w$, one can obtain a reduced word $\s w'$ by repeatedly replacing $\alpha\alpha^{-1}$ or $\alpha^{-1}\alpha$ with the empty word.
\begin{rmk}\label{rmk: traversal in N1 has a unique string}
Given a traversal $\f t:=\alpha_n\cdots\alpha_1$ in $\C N[1]$, an $\s A$-word can be associated with $\f t$ in the following way: $\f W(\f t):=e$ if $|\f t|=0$, otherwise $\f W(\f t):=\f W(\alpha_n)\cdots\f W(\alpha_1)$, where $\f W(\alpha_i):=\begin{cases}
    F(\alpha_i) &\text{ if }\alpha_i\in\C N[1]\cup\C N[1]^{-1};\\
    e &\text{ otherwise.}
\end{cases}$
Let $\f W^\R{red}(\f t)$ denote the reduced word obtained from $\f W(\f t)$. Let $\f t_1$ and $\f t_2$ be traversals with $\f s(\f t_1)=\f s(\f t_2)$ and $\f e(\f t_1)=\f e(\f t_2)$. Then $\f W^\R{red}(\f t_1)=\f W^\R{red}(\f t_2)$, for otherwise there would be a cycle in the underlying graphs of both $T^1$ and $T^2$, a contradiction to them being trees.
\end{rmk}

\begin{rmk}\label{rmk: E1 asymmetric}
The definition of $\C E^1$ is inherently asymmetric. To understand this, let $\C P^{12}$ and $\C P^{21}$ denote the pullback quivers used as arguments while constructing the pullback networks associated with the pairs $(M_1,M_2)$ and $(M_2,M_1)$ respectively. Then the vertex assignment $(n,m)\mapsto(m,n)$ induces an isomorphism between $\C P^{12}$ and $\C P^{21}$. However, this isomorphism does not necessarily induce an isomorphism between the respective pullback networks for it may not induce an isomorphism between the sets of edges as shown in \Cref{exmp: E1 asymmetric}.
\end{rmk}

\begin{exmp}\label{exmp: E1 asymmetric}
Consider $T^1$, $T^2$ and $\C N[1]$ as shown in \Cref{fig:situation edge N1 1}, where $\{(n',m'),(n'',m')\}\in\C E^1$, whereas the pullback network associated with the pair $(M_2,M_1)$ does not contain the edge $\{(m',n'),(m',n'')\}$.
\end{exmp}

A \emph{triangle} is a subnetwork of $\C N[1]$ containing three vertices whose underlying graph is a clique. Denote the set of triangles in $\C N[1]$ by $\triangle$. In view of \Cref{rmk: atmost one arrow or edge between two vertices}, given a collection $\C V$ of three vertices in $\C N[1]$, there is at most one triangle in $\ang{\C V}$. Recall from \Cref{rmk: traversal in N1 has a unique string} that the associated $\s A$-word for an edge and an arrow is the empty word and a non-empty word respectively. This guarantees that there is always an odd number of edges in a triangle. There are only four types of triangles; all types are shown in \Cref{fig:triangle}.

\begin{figure}[H]
\centering
\begin{subfigure}{0.3\linewidth}
\begin{tikzcd}[column sep=small]
                               & {(n,m)} \arrow[ld, "{(a,b)}"'] \arrow[rd, "{(c,b)}"] &           \\
{(n'',m')} \arrow[rr, no head] &                                                      & {(n',m')}
\end{tikzcd}
\caption{}
\label{subfig: triangle type I}
\end{subfigure}
\hspace{2mm}
\begin{subfigure}{0.3\linewidth}
\begin{tikzcd}[column sep=small]
& {(n,m)} &                                   \\
{(n',m')} \arrow[ru, "{(a,b)}"] \arrow[rr, no head] &         & {(n',m'')} \arrow[lu, "{(a,c)}"']
\end{tikzcd}
\caption{}
\label{subfig: triangle type II}
\end{subfigure}\\
\vspace{5mm}
\begin{subfigure}{0.3\linewidth}
\begin{tikzcd}[column sep=small]
                             & {(n,m)} \arrow[ld, no head] \arrow[rd, no head] &           \\
{(n,m')} \arrow[rr, no head] &                                                 & {(n,m'')}
\end{tikzcd}
\caption{}
\label{subfig: triangle type III}
\end{subfigure}
\hspace{2mm}
\begin{subfigure}{0.3\linewidth}
\begin{tikzcd}[column sep=small]
                             & {(n,m)} \arrow[ld, no head] \arrow[rd, no head] &           \\
{(n',m)} \arrow[rr, no head] &                                                 & {(n'',m)}
\end{tikzcd}
\caption{}
\label{subfig: triangle type IV}
\end{subfigure}

    \caption{Four types of triangles}
    \label{fig:triangle}
\end{figure}
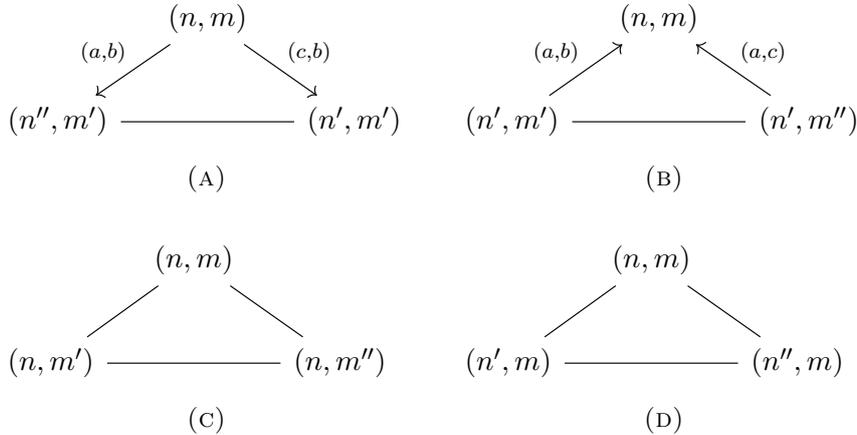

We will see now that these triangles have an interesting structure, a useful result which will be used in \S~\ref{sec: R2 free from complete}.

\begin{prop}\label{prop: triangle structure 1}
Let $\C V_1:=\{(n_1,m_1),(n_2,m_2),(n_3,m_3)\}$, $\C V_2:=\{(n_1,m_1),(n_2,m_2),(n_4,m_4)\}$ be subsets of $\C N[1]_0$ such that $\ang{\C V_1},\ang{\C V_2}\in\triangle$. Then there exists a link from $(n_3,m_3)$ to $(n_4,m_4)$.
\end{prop}
\begin{proof}
We mention the proofs of two cases; the proofs for the rest of the cases are similar.

\textbf{Case I.} Suppose $\ang{V_1}$ and $\ang{\C V_2}$ are triangles of types A and B respectively (see \Cref{fig:triangle}), and share a common arrow.

As shown in \Cref{fig: case 1 triangle structure 1}, we have $(n_1,m_1)=(n,m)$, $(n_2,m_2)=(n'',m')$, $(n_3,m_3)=(n',m')$ and $(n_4,m_4)=(n,m'')$ with $n,n',n'',m,m''$ pairwise distinct. It is obvious to see that the arrows $n\xrightarrow{a}n'$ and $m''\xrightarrow{d}m'$ satisfy $F_1(a)=F_2(d)$, thus implying the existence of the arrow $(n,m'')\xrightarrow{(a,d)}(n',m')$.\\

\textbf{Case II.} Suppose $\ang{V_1}$ and $\ang{\C V_2}$ are triangles of type C (see \Cref{fig:triangle}).

As shown in \Cref{fig: case 2 triangle structure 1}, we have $(n_1,m_1)=(n,m)$, $(n_2,m_2)=(n,m')$, $(n_3,m_3)=(n,m'')$ and $(n_4,m_4)=(n,m''')$ with $n,m,m',m'',m'''$ pairwise distinct. The definition of $\C E^1$ and the uniqueness of paths between two vertices in a tree forces the existence of a vertex $\overline m$ and arrows $a,b,c,d$ in $T^2$ with $F_2(a)=F_2(b)=F_2(c)=F_2(d)$ as shown in \Cref{fig: case 2 codomain triangle structure 1}, thereby concluding the existence of an edge between $(n,m'')$ and $(n,m''')$ in $\C N[1]$.
\end{proof}

\begin{figure}[H]
    \centering
    \begin{subfigure}{0.45\textwidth} 
        \[\begin{tikzcd}[row sep=huge]
{(n,m)} \arrow[d, "{(a,b)}"'] \arrow[rd, "{(c,b)}"'{pos=0.85,rotate=-35}] \arrow[r, no head] & {(n,m'')} \arrow[d, "{(c,d)}"] \arrow[ld, color=red, "{(a,d)}"'{pos=0.1,rotate=35}] \\
{(n',m')} \arrow[r, no head]                                           & {(n'',m')}                                          
\end{tikzcd}\]
        \caption{Case I}
        \label{fig: case 1 triangle structure 1}
    \end{subfigure}
    \hfill
    \begin{subfigure}{0.45\textwidth}
        \[\begin{tikzcd}[row sep=huge]
{(n,m)} \arrow[d, no head] \arrow[rd, no head] \arrow[r, no head] & {(n,m''')} \arrow[d, no head] \arrow[ld, no head,color=red] \\
{(n,m'')} \arrow[r, no head]                                      & {(n,m')}                                         
\end{tikzcd}\]
        \caption{Case II}
        \label{fig: case 2 triangle structure 1}
    \end{subfigure}
    \caption{The structure of triangles used in the proof of \Cref{prop: triangle structure 1}}
    \label{fig:triangle structure 1}
\end{figure}
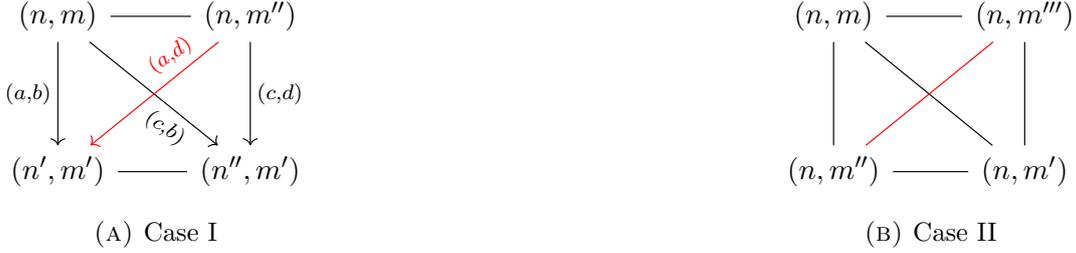

\begin{figure}[H]
    \[\begin{tikzcd}
m \arrow[rrrdd, "a"'] &  & m' \arrow[rdd, "b"] &       & m'' \arrow[ldd, "c"'] &  & m''' \arrow[llldd, "d"] \\
                     &  &                     &       &                      &  &                         \\
                     &  &                     & \overline m &                      &  &                        
\end{tikzcd}\]
    \caption{Subtree of $T^2$ used in Case II of \Cref{prop: triangle structure 1}}
    \label{fig: case 2 codomain triangle structure 1}
\end{figure}
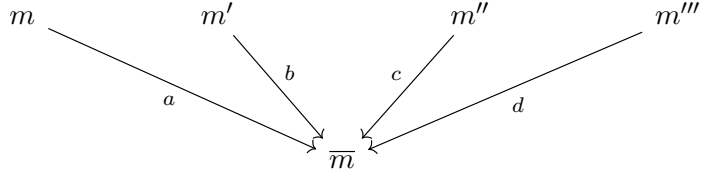


Define a relation $\sim$ on $\triangle$ as follows: given $\C T_1,\C T_2\in\triangle$, say that $\C T_1\sim\C T_2$ if they share a common link. Clearly, $\sim$ is reflexive and symmetric. Let $\approx$ be the transitive closure of $\sim$. Therefore, $\approx$ is an equivalence relation on $\triangle$. The following proposition shows an interesting property of an $\approx$-equivalence class.

\begin{prop}\label{prop: triangle structure complete graph}
Consider an $\approx$-equivalence class $\Gamma$. Denote by $\Gamma_0$ the set of vertices that appear in $\Gamma$. Then the underlying graph of $\ang{\Gamma_0}$ is a clique.
\end{prop}
\begin{proof}
Let $(n,m)$ and $(n',m')$ be distinct vertices in $\Gamma_0$. We will show that there is a link from $(n,m)$ to $(n',m')$. There is a sequence
\begin{equation}\label{eqn: seq of triangles}
\C T^1\sim\C T^2\sim\cdots\sim\C T^L
\end{equation} of pairwise distinct  triangles in $\Gamma$ such that $(n,m)\in\C T^1_0\setminus\C T^2_0$ and $(n',m')\in\C T^L_0\setminus\C T^{L-1}_0$. We will prove the result by induction on the length $L$ of such a sequence.

\textbf{Base Case.} $L=1.$

Here, the triangle $\C T^1$ has $(n,m)$ and $(n',m')$ as two of its vertices. Therefore, there exists a link from $(n,m)$ to $(n',m')$.\\

\textbf{Induction Step.} Suppose the result is true for all $1\leq L\leq k$. Assume that $L=k+1$.

Note that for $(n'',m'')\in \bigcup_{l=1}^{k+1}\C T^l_0\setminus\{(n,m),(n',m')\}$, the length of a sequence of triangles for the pair $(n,m)$ and $(n'',m'')$ as in Sequence \eqref{eqn: seq of triangles} is less than $k+1$, and therefore, by the induction hypothesis, there exists a link from $(n,m)$ to $(n'',m'')$. Thus $\ang{\C V}\in\triangle$, where $\C V:=\{(n,m)\}\cup\C T^{(k+1)}_0\setminus\{(n',m')\}$. Since $\ang{\C V},\C T^{k+1}\in\triangle$ share a common link, \Cref{prop: triangle structure 1} yields a link from $(n,m)$ to $(n',m')$.
\end{proof}

Even though sign-flips are encoded using edges in $\C N[1]$, the signs are not ``visible'' in the network. To give them visibility, we define the \emph{$2$-covering network} associated with the pair $(M_1,M_2)$ of generalised tree modules.

\begin{defn}
The \emph{$2$-covering network} $\C N[2]:=(\C N[2]_0,\C N[2]_1,s^2,t^2,\C E^2)$ associated with a pair of generalised tree modules $(M_1,M_2)$ is defined as follows:
\begin{align*}
    \C N[2]_l&:=\C N[1]_l\times\{-1,1\}\text{ for }l\in\{0,1\},\\
    s^2((a,b,i))&:=(s^1(a,b),i)\text{ for }(a,b,i)\in\C N[2]_1,\\
    t^2((a,b,i))&:=(t^1(a,b),i)\text{ for }(a,b,i)\in\C N[2]_1,\\
    \C E^2&:=\{\{(n,m,i),(n',m',-i)\}\mid\{(n,m),(n',m')\}\in\C E^1, i\in\{-1,1\}\}.
\end{align*}
\end{defn}

\begin{rmk}
If $\R{char}(\C K)=2$, then $\C N[2]=\C N[1]$, and hence we assumed that $\R{char}(\C K)\neq2$.
\end{rmk}

The next remark explains why the $2$-covering network $\C N[2]$ is a cover of $\C N[1]$.

\begin{rmk}\label{rmk: proj between network}
It is trivial to note that there is a canonical projection $\pi:\C N[2]\to\C N[1]$ given by
\begin{equation*}
    \pi((x,y,i)):=(x,y)\text{ for every }(x,y,i)\in\C N[2]_0\cup\C N[2]_1,
\end{equation*}
satisfying $\pi\circ s^2=s^1\circ\pi$, $\pi\circ t^2=t^1\circ\pi$, and mapping $\C E^2$ to $\C E^1$ surjectively. Since $|\pi^{-1}(\alpha)|=2$ for each $\alpha\in\C N[1]_0\cup\C N[1]_1\cup\C E^1$, $\C N[2]$ is indeed a $2$-cover of $\C N[1]$.
\end{rmk}

\begin{rmk}\label{rmk: similar remarks for N2}
Using \Cref{rmk: proj between network}, it is readily verified that the conclusions of Remarks \ref{rmk: atmost one arrow or edge between two vertices} and \ref{rmk: traversal in N1 has a unique string} also hold true for $\C N[2]$.
\end{rmk}

\begin{defn}
Let $\C T\in\triangle$. Call the subnetwork $\pi^{-1}(\C T)$ a \emph{hexagon} in $\C N[2]$. \Cref{fig: triangle and hexagon} shows a hexagon in $\C N[2]$ corresponding to a triangle in $\C N[1]$.
\end{defn}

\begin{figure}[H]
    \[\begin{tikzcd}[column sep=small]
                              &                                                      &            &  &                                & {(n,m,1)} \arrow[ld, "{(a,b,1)}"'] \arrow[rd, "{(c,b,1)}"]    &                                 \\
                              & {(n,m)} \arrow[ld, "{(a,b)}"'] \arrow[rd, "{(c,b)}"] &            &  & {(n',m',1)} \arrow[d, no head] &                                                               & {(n'',m',1)}                    \\
{(n',m')} \arrow[rr, no head] &                                                      & {(n'',m')} &  & {(n'',m',-1)}                  &                                                               & {(n',m',-1)} \arrow[u, no head] \\
                              &                                                      &            &  &                                & {(n,m,-1)} \arrow[lu, "{(c,b,-1)}"] \arrow[ru, "{(a,b,-1)}"'] &                                
\end{tikzcd}\]
    \caption{A triangle in $\C N[1]$ and the corresponding hexagon in $\C N[2]$}
    \label{fig: triangle and hexagon}
\end{figure}
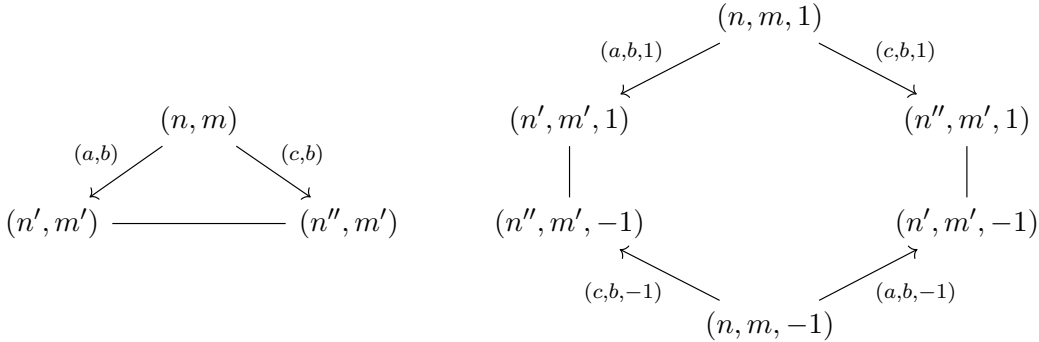

\section{Generalised graph maps}\label{sec: ggm}
If $M_1=F_{1\lambda}(V_{T^1})$ and $M_2=F_{2\lambda}(V_{T^2})$ are tree modules, recall from \Cref{thm: c-b} the description of a graph map $(\Tfac,\Tim,\Phi)$ from $M_1$ to $M_2$. The quiver isomorphism $\Phi$ can be treated as a subnetwork $\C M:=(\C M_0,\C M_1,\C{E_M})$ of $\C N[2]$ satisfying the following properties for each $(n,m,1)\in \C M_0$:
\begin{itemize}
    \item for every $n'\xrightarrow{a}n$ in $T^1$, there exists a unique $m'\xrightarrow{b}m$ in $T^2$ so that $(n',m',1)\in \C M_0$ and 
    $$((n',m',1)\xrightarrow{(a,b,1)}(n,m,1))\in \C M_1;$$
    \item for every $m\xrightarrow{b}m'$ in $T^2$, there exists a unique $n\xrightarrow{a}n'$ in $T^1$ so that $(n',m',1)\in \C M_0$ and
    $$((n,m,1)\xrightarrow{(a,b,1)}(n',m',1))\in \C M_1.$$
\end{itemize}

On the other hand, given a subnetwork $\C M:=(\C M_0,\C M_1,\C{E_M})$ of $\C N[2]$, we can define a $\C K$-linear map $\C {H_M}:M_1\to M_2$ as follows:
\begin{equation}\label{map from M1 to M2}
    \C{ H_{M}}(v_n):=\sum_{(n,m,j)\in\C M_0}jw_m.
\end{equation}
The $\C K$-linear map $\C{ H_{M}}$ defined above is not necessarily a homomorphism as \Cref{ex: blocked reln motivation} shows.

\begin{exmp}\label{ex: blocked reln motivation}
Consider the generalised tree modules $M_1:=F_\lambda(V_{T^1})$ and $M_2:=F_\lambda(V_{T^2})$ over $\C K\B A_2$, where the trees $T^1$ and $T^2$ are given in Figure \ref{fig:blocked relation motivation} with $F_1(1)=F_2(4)=\B 1$, $F_1(2)=F_1(3)=F_2(5)=\B2$ and $F_1(a)=F_1(b)=F_2(c)=\B a$.
The $2$-covering network $\C N[2]$ is shown in \Cref{fig: N2 for an example}. The $\C K$-linear map corresponding to the induced subnetwork $\ang{\C M_0}$ (see \Cref{map from M1 to M2}), where $\C M_0:=\{(2,5,1),(1,4, 1),(3,5,1)\}$, is not a homomorphism since $\B a\cdot\C{ H_{M}}(v_1)=w_5\neq 2w_5=\C{ H_{M}}(\B a\cdot v_1).$
\end{exmp}

\begin{minipage}{0.45\linewidth}
    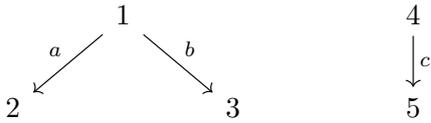
\begin{figure}[H]
    \[\begin{tikzcd}
  & 1 \arrow[ld, "a"'] \arrow[rd, "b"] &   &  & 4 \arrow[d, "c"] \\
2 &                                    & 3 &  & 5                
\end{tikzcd}\]
    \caption{Two trees $T^1$ and $T^2$ used in \Cref{ex: blocked reln motivation}}
    \label{fig:blocked relation motivation}
\end{figure}
\end{minipage}
\hfill
\begin{minipage}{0.55\linewidth}
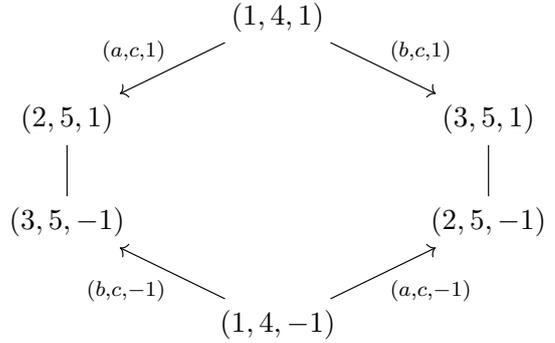
\begin{figure}[H]
    \[\begin{tikzcd}
                             & {(1,4,1)} \arrow[ld, "{(a,c,1)}"'] \arrow[rd, "{(b,c,1)}"]    &                              \\
{(2,5,1)} \arrow[d, no head] &                                                               & {(3,5,1)} \arrow[d, no head] \\
{(3,5,-1)}                   &                                                               & {(2,5,-1)}                   \\
                             & {(1,4,-1)} \arrow[lu, "{(b,c,-1)}"] \arrow[ru, "{(a,c,-1)}"'] &                             
\end{tikzcd}\]
    \caption{Network $\C N[2]$ in \Cref{ex: blocked reln motivation}}
    \label{fig: N2 for an example}
\end{figure}
\end{minipage}
\vspace{2mm}

Therefore, when $M_1,M_2$ are generalized tree modules, we will define certain subnetworks of $\C N[2]$ in this section for which the corresponding map, as defined in \Cref{map from M1 to M2}, is a homomorphism. Such subnetworks will be called ``generalised graph maps'' (\Cref{defn: ggm}) since they are a generalisation of graph maps between two tree modules. Contrary to what one might expect, the (finite) set of such maps will be a generating set (\Cref{thm: main 1}) for $\Hom$ and it may fail to be a basis (see \Cref{exmp: ggm}).

Note that there are two kinds of properties of a subnetwork of $\C N[2]$ corresponding to a graph map between two tree modules--existence and uniqueness. Keeping in mind the ``sign-flip homomorphisms'' as shown in \Cref{exmp: motivation of flip}, the existence conditions are generalised and captured in a \emph{complete subnetwork} as defined below.

\begin{defn}\label{defn: complete subnetwork}
A subnetwork $\C M:=(\C M_0,\C M_1,\C{E_M})$ of $\C N[2]$ is called \emph{complete} if the following hold for each $(n,m,j)\in\C M_0$.
\begin{enumerate}
    \item For every $n'\xrightarrow{a}n$ in $T^1$, at least one of the following holds:
    \begin{enumerate}
        \item there exists $m'\xrightarrow{b} m$ in $T^2$ such that $$((n',m',j)\xrightarrow{(a,b,j)}(n,m,j))\in\C M_1;$$

        \item there exist $n\xleftarrow{a}n'\xrightarrow{c}n''$ in $T^1$ with $F_1(a)=F_1(c)$, and $(n'',m,-j)\in\C M_0$ such that $$\{(n,m,j),(n'',m,-j)\}\in\C{E_M}.$$
    \end{enumerate}   

    \item For every $m\xrightarrow{b}m'$ in $T^2$, at least one of the following holds: 
    \begin{enumerate}
        \item there exists $n\xrightarrow{a}n'$ in $T^1$ such that $$((n,m,j)\xrightarrow{(a,b,j)}(n',m',j))\in\C M_1;$$

        \item there exist $m\xrightarrow{b}m'\xleftarrow{d}m''$ in $T^2$ with $F_2(b)=F_2(d)$, and $(n,m'',-j)\in\C M_0$ such that $$\{(n,m,j),(n'',m,-j)\}\in\C{E_M}.$$
    \end{enumerate}     
\end{enumerate}
\end{defn}

\begin{exmp}\label{exmp: complete network}
Continuing from \Cref{ex: blocked reln motivation}, it is easy to verify that the non-empty complete subnetworks of $\C N[2]$ (see \Cref{fig: N2 for an example}) are precisely the subnetworks which have at least two vertices in each of its connected components.
\end{exmp}

Here is a simple observation on the completeness of two subnetworks of $\C N[2]$ with the same set of vertices.

\begin{rmk}\label{rmk: super network is complete}
Let $\C M^1$ be a subnetwork of $\C N[2]$ and $\C M^2$ be a subnetwork of $\C M^1$  such that $\C M^2_0=\C M^1_0$. If $\C M^2$ is complete, then so is $\C M^1$.
\end{rmk}

The uniqueness conditions for a subnetwork corresponding to a graph map between two tree modules in generalised in the following way: for Condition (1) of \Cref{defn: complete subnetwork}, at most one of Conditions (1a) or (1b) should hold, at most one arrow $b$ should satisfy Condition (1a), and at most one edge should satisfy Condition (1b). Similar is the case for Condition (2). We capture this information in the form of ``blocked traversals'' and look at those subnetworks that do not contain such blocked traversals. Suprisingly, all such blocked traversals are adjacent links of a hexagon in $\C N[2]$. Therefore, define 
\begin{align}
    \label{defn: R1}\C R[1]&:=\{(n_1,m_1)--(n_2,m_2)--(n_3,m_3)\in\f T(\C N[1])\mid\langle\{(n_k,m_k)\mid k\in\{1,2,3\}\}\rangle\in\triangle\};\text{ and}\\
    \label{defn: R2}\C R[2]&:=\{\f t\in\f T(\C N[2])\mid\f \pi(\f t)\in\C R[1]\}.
\end{align}

\begin{defn}\label{defn: Rj free}
For $j\in\{1,2\}$, a subnetwork $\C M$ of $\C N[j]$ is said to be \emph{$\C R[j]$-free} if there is no traversal $\f t$ of $\C M$ that lies in $\C R[j]$.
\end{defn}

The following result, which we mention without proof, captures the discussion in the last paragraph.

\begin{prop}\label{prop: uniqueness=Rfree}
Let $\C M:=(\C M_0,\C M_1,\C{E_M})$ be a complete subnetwork of $\C N[2]$. Then $\C M$ is $\C R[2]$-free if and only if for every $(n,m,j)\in\C M_0$, both the following hold:
\begin{itemize}
    \item for each  $(n'\xrightarrow{a}n)\in T^1_1$, exactly one of Conditions (1a) and (1b) of \Cref{defn: complete subnetwork} hold, at most one arrow $b$ satisfies Condition (1a), and at most one edge satisfies Condition (1b); and
    \item for each $(m\xrightarrow{b}m')\in T^2_1$, exactly one of Conditions (2a) and (2b) hold, at most one arrow $b$ satisfies Condition (2a), and at most one edge satisfies Condition (2b).
\end{itemize}
\end{prop}
\begin{rmk}\label{rmk: connected component of complete R free}
Let $\C M$ be a complete $\C R[2]$-free subnetwork of $\C N[2]$ and $\C M'$ be a connected component of $\C M$. Then $\C M'$ is also a complete $\C R[2]$-free subnetwork of $\C N[2]$.
\end{rmk}

\begin{exmp}\label{exmp: complete R-free subnetwork}
Continuing from \Cref{exmp: complete network}, any subnetwork containing exactly two vertices in each connected component characterises all the non-empty complete $\C R[2]$-free subnetworks of $\C N[2]$ (see \Cref{fig: N2 for an example}).
\end{exmp}

We have been aiming at the following result.

\begin{prop}\label{prop: complete and Rfree is homo}
Let $\C M:=(\C M_0,\C M_1,\C{E_M})$ be a complete and $\C R[2]$-free subnetwork of $\C N[2]$. Then the $\C K$-linear map $\C{ H_{M}}:M_1\to M_2$ as defined in \Cref{map from M1 to M2} is a homomorphism.
\end{prop}
\begin{proof}
As mentioned at the beginning of \S~\ref{sec: cov network}, we will denote the $\C K$-basis elements of $M_1$ and $M_2$ by $\{v_n\}_{n\in T^1_0}$ and $\{w_m\}_{m\in T^2_0}$ respectively. It is sufficient to show that $\alpha\cdot\C{H_M}(v_n)=\C{H_M}(\alpha\cdot v_n)$ for every $n\in T^1_0$ and every $\alpha\in  Q_1$ by showing that for each $m\in T^2_0$, the coefficient of $w_m$ is identical in $\alpha\cdot\C{H_M}(v_n)$ and $\C{H_M}(\alpha\cdot v_n)$ with respect to the chosen bases. Let $n_0\in T^1_0$ and $m_0\in T^2_0$ be arbitrarily chosen and fixed.

Suppose $\C{H_M}(v_{n_0})=\sum_{(n_0,m,j)\in\C M_0}jw_m=:\sum_{k=1}^Lj_kw_{m_k}$, where each $j_k\in\{-1,1\}$. Let $\mathscr{I}:=\{1,\cdots,L\}$. Partition $\mathscr I$ into three sets $\mathscr I_1$, $\mathscr I_2$ and $\mathscr I_3$ as follows:
\begin{align*}
    \mathscr{I}_1&:=\{k\in\mathscr I\mid\text{the coefficient of $w_{m_0}$ in }\alpha\cdot w_{m_k}\text{ is zero}\};\\
    \mathscr{I}_2&:=\{k\in\mathscr I\mid\text{ there exists $k'\in\mathscr{I}$ such that }\{(n,m_k,j_k),(n,m_{k'},j_{k'})\}\in\C {E_M}\}\setminus\mathscr{I}_1;\text{ and}\\
    \mathscr{I}_3&:=\mathscr{I}\setminus(\mathscr{I}_1\cup\mathscr{I}_2).
\end{align*}
Using the uniqueness of an edge in Condition (2b) of \Cref{defn: complete subnetwork}, in view of \Cref{prop: uniqueness=Rfree}, for each $k\in\mathscr{I}_2$, there exists unique $k'\in\mathscr{I}_2$ such that $j_{k'}=-j_k$ and $\{(n,m_k,j_k),(n,m_{k'},j_{k'})\}\in\C {E_M}$. Therefore, the coefficient of $w_{m_0}$ in $\alpha\cdot(\sum_{k\in\mathscr{I}_2}j_kw_{m_k})$ is zero, which gives that the coefficient of $w_{m_0}$ in $\alpha\cdot(\sum_{k\in\mathscr{I}_3}j_kw_{m_k})$ is same as that in $\alpha\cdot({\sum_{k\in\mathscr{I}}j_kw_{m_k}})$.

Now, suppose $\alpha\cdot v_{n_0}=:\sum_{k=1}^{L'}v_{n_k}$ with the coefficient of $w_{m_0}$ in $\C{H_M}(v_{n_k})$ being $j'_k$. Let $\mathscr{J}:=\{1,\cdots,L'\}$. Partition $\mathscr J$ into three sets $\mathscr J_1$, $\mathscr J_2$ and $\mathscr J_3$ as follows:
\begin{align*}
    \mathscr{J}_1&:=\{k\in\mathscr{J}\mid j'_k=0\};\\
    \mathscr{J}_2&:=\{k\in\mathscr J\mid\text{ there exists $k'\in\mathscr{J}$ such that }\{(n_k,m,j'_k),(n_{k'},m,j'_{k'})\}\in\C {E_M}\}\setminus\mathscr{I}_1;\text{ and}\\
    \mathscr{J}_3&:=\mathscr{J}\setminus(\mathscr{J}_1\cup\mathscr{J}_2).
\end{align*}
Using the uniqueness of an edge in Condition (1b) of \Cref{defn: complete subnetwork}, in view of \Cref{prop: uniqueness=Rfree}, for each $k\in\mathscr{J}_2$, there exists a unique $k'\in\mathscr{J}_2$ such that $j'_{k'}=-j'_k$ and $\{(n_k,m,j'_k),(n_{k'},m,j'_{k'})\}\in\C {E_M}$. Therefore, the coefficient of $w_{m_0}$ in $\C{H_M}(\sum_{k\in\mathscr{J}_2}v_{n_k})$ is zero, which gives that the coefficient of $w_m$ in $\C{H_M}(\sum_{k\in\mathscr{J}_3}v_{n_k})$ is same as that in $\C{H_M}(\sum_{k\in\mathscr{J}}v_{n_k})$.

It is clear from the discussion in the above two paragraphs that in order to complete the proof of the proposition, it is sufficient to show that $\sum_{k\in\mathscr I_3}j_k=\sum_{k\in\mathscr J_3}j'_k$.

For each $k\in\mathscr{I}_3$, we have that $(n_0,m_k,j_k)\in\C M_0$ and $(m_k\xrightarrow{a}m_0)\in T^2_1$ with $F_2(a)=\alpha$, and Condition (2b) is not satisfied, for otherwise $k\in\mathscr I_2$, a contradiction to the hypothesis that $\C M$ is $\C R[2]$-free. Since $\C M$ is complete and $\C R[2]$-free, we have that Condition (2a) is satisfied for $k\in\mathscr{I}_3$ by a unique arrow (cf. \Cref{prop: uniqueness=Rfree}), thus providing an injection $\mathscr{F}:\mathscr{I}_3\to\mathscr{J}_3$ such that $j'_{\mathscr{F}(k)}=j_k$. Similarly, using uniqueness in  Condition (1a) of \Cref{defn: complete subnetwork}, there is an injection $\mathscr{F'}:\mathscr{J}_3\to\mathscr{I}_3$ such that $j_{\mathscr{F}(k)}=j'_k$. Therefore, there is a bijection between the sets $\mathscr{I}_3$ and $\mathscr{J}_3$ yielding an equality $\sum_{k\in\mathscr{I}_3}j_k=\sum_{k\in\mathscr{J}_3}j'_k$, as required.
\end{proof}

The coefficient of $w_m$ in $\C {H_M}(v_n)$ is $0$ if both $(n,m,1)$ and $(n,m,-1)$ are in $\C M_0$. To avoid this redundancy, it is desirable to consider those subnetworks $\C M'$ of $\C N[2]$ for which if $(n,m,j)\in\C M_0$ then $(n,m,-j)\notin\C M_0$, which we call \emph{involution-free} subnetworks. Also recall that, in view of \Cref{rmk: connected component of complete R free}, it is sufficient to work with connected subnetworks. Finally, we are ready to introduce the long-awaited concept of generalised graph maps from $M_1$ to $M_2$.

\begin{defn}\label{defn: ggm}
A \emph{generalised graph map} $\C G$ from $M_1$ to $M_2$ is a non-empty connected involution-free complete $\C R[2]$-free subnetwork of $\C N[2]$. We denote the homomorphism obtained from a generalised graph map $\C G$ by $\C{H_G}$.
\end{defn}

The following remarks are simple observations about generalised graph maps.

\begin{rmk}
The set of generalised graph maps from $M_1$ to $M_2$ is finite.
\end{rmk}

\begin{rmk}\label{rmk: -generalised graph map}
If $\C G:=(\C G_0,\C G_1,\C{E_G})$ is a generalised graph map then so is $-\C G:=(-\C G_0,-\C G_1,\C{E_{-G}})$, where
\begin{align*}
    -\C G_i&:=\{(x,y,j)\in\C N[2]_i\mid(x,y,-j)\in\C G_i\}\text{ for }i\in\{0,1\}; \text{ and}\\
    \C{E_{-G}}&:=\{\{(n,m,j),(n',m',-j)\}\in\C E^2\mid\{(n,m,-j),(n',m',j)\}\in\C{E_G}\}.
\end{align*}
Also, the corresponding homomorphisms satisfy $$\C{H_{-G}}=-\C{ H_G}.$$
\end{rmk}

\begin{exmp}\label{exmp: ggm}
Continuing from \Cref{exmp: complete R-free subnetwork}, the generalised graph maps are $\pm \C G_1$, $\pm\C G_2$, $\pm \C G_3$ as defined below:
\begin{align*}
    \C G_1:\  (1,4,1)\xrightarrow{(a,c,1)}(2,5,1);\ \ 
    \C G_2:\  (1,4,1)\xrightarrow{(b,c,1)}(3,5,1);\ \ 
    \C G_3:\  (2,5,1)\text{ ----- }(3,5,-1).
\end{align*}
The homomorphisms corresponding to the generalised graph maps are $\pm\C H_{\C G_1}$, $\pm\C H_{\C G_2}$, $\pm\C H_{\C G_3}$ as defined below:
\begin{align*}
    \C H_{\C G_1}:\ v_1\mapsto w_4,\ v_2\mapsto w_5,\ v_3\mapsto0;\ \ 
    \C H_{\C G_2}:\ v_1\mapsto w_4,\ v_2\mapsto 0,\ v_3\mapsto w_5;\ \ 
    \C H_{\C G_3}:\ v_1\mapsto0,\ v_2\mapsto w_5,\ v_3\mapsto -w_5.
\end{align*}
Note that $\C H_{\C G_1}-\C H_{\C G_2}=\C H_{\C G_3}$, which implies that the set $\{\C H_{\C G_1},\C H_{\C G_2},\C H_{\C G_3}\}$ is linearly dependent.
\end{exmp}

\section{Carving out a complete $\C R[2]$-free network from a complete network}\label{sec: R2 free from complete}
Let $\C V$ be a subset of $\C N[1]_0$ such that $\C M:=\ang{\pi^{-1}(\C V)}$ of $\C N[2]$ is complete. In this section, we describe an algorithm to obtain a complete $\C R[2]$-free subnetwork $\C M'$ of $\C M$ satisfying $\C M'_0=\C M_0=\pi^{-1}(\C V)$.

Recall the relation $\approx$ on the set $\triangle$ of triangles in $\C N[1]$ defined before \Cref{prop: triangle structure 1}. We denote the set of vertices in an $\approx$-equivalence class $\Gamma$ in $\triangle$ by $\Gamma_0$.

\begin{defn}\label{defn: R system}
Say that a subnetwork $\C S[1]$ of $\ang{\C V}$ is an \emph{$\C R[1]$-system relative to $\C V$} if $\C S[1]=\ang{\Gamma_0\cap\C V}$ for some $\approx$-equivalence class $\Gamma$ satisfying $|\Gamma_0\cap\C V|\geq3$. Say that a subnetwork $\C S[2]$ of $\C M$ is an \emph{$\C R[2]$-system relative to $\C V$} if $\C S[2]=\pi^{-1}(\C S[1])$ for some $\C R[1]$-system $\C S[1]$ relative to $\C V$.
\end{defn}

We shall drop the phrase ``relative to $\C V$'' when the set $\C V$ is clear from context. The following remarks shed more light on the structure of $\C R[1]$-systems and $\C R[2]$-systems.

\begin{rmk}\label{rmk: 2 length traversal in R system blocked}
Using Definitions \eqref{defn: R1} and \eqref{defn: R2}, we have that every two length traversal in an  $\C R[1]$-system (resp. $\C R[2]$-system) is in $\C R[1]$ (resp. $\C R[2]$).
\end{rmk}

\begin{rmk}\label{rmk: structure R system-clique}
Using \Cref{prop: triangle structure complete graph}, we have that the underlying graph of an $\C R[1]$-system is a clique. Therefore, for any set $\C V'$ of three vertices of an $\C R[1]$-system, the subnetwork $\ang{\pi^{-1}(\C V')}$ is a hexagon in the corresponding $\C R[2]$-system. 
\end{rmk}

\begin{rmk}\label{rmk: structure R system-sharing}
Any two triangles in an $\C R[1]$-system either share a common link or a common vertex or nothing. Correspondingly, any two hexagons in an $\C R[2]$-system either share a pair of antipodal links or a pair of antipodal vertices or nothing.
\end{rmk}

We wish to obtain a complete $\C R[2]$-free subnetwork of $\C N[2]$ containing $\pi^{-1}(\C V)$ by obtaining an $\C R[2]$-free subnetwork of each $\C R[2]$-system containing all its vertices. In view of \Cref{rmk: 2 length traversal in R system blocked}, obtaining a complete $\C R[2]$-free subnetwork of a complete subnetwork of $\C N[2]$ keeping the set of vertices intact is equivalent to finding a perfect matching(=a set $E$ of edges such that no two edges in $E$ are incident on a common vertex, but each vertex is incident on some edge in $E$) of the underlying graph of each of its $\C R[2]$-systems.

\begin{prop}\label{prop: complete R free from complete involution invariant}
Let $\C V\subseteq\C N[1]_0$ and $\C M:=\ang{\pi^{-1}(\C V)}$ be a complete subnetwork of $\C N[2]$. Then there exists a complete $\C R[2]$-free subnetwork $\C M'$ of $\C M$ such that $\C M'_0=\C M_0$.
\end{prop}
\begin{proof}
Consider an $\C R[2]$-system $\C S[2]$ relative to $\C V$ with $\C S[2]_0:=\{(n_1,m_1),(n_2,m_2),\cdots,(n_k,m_k)\}$ for some $k\geq3$. Consider the following partition $\C P$ of $\C S[2]_0$: 
\begin{equation*}
    \C P:=
    \begin{cases}
    \{\{(n_1,m_1),(n_2,m_2)\},\{(n_3,m_3),(n_4,m_4)\},\cdots,\{(n_{k-2},m_{k-2}),(n_{k-1},m_{k-1}),(n_k,m_k)\}\}&\text{ if $k$ is odd};\\
    \{\{(n_1,m_1),(n_2,m_2)\},\{(n_3,m_3),(n_4,m_4)\},\cdots,\{(n_{k-1},m_{k-1}),(n_k,m_k)\}\}&\text{ if $k$ is even}.
    \end{cases}
\end{equation*}
In view of \Cref{rmk: structure R system-clique}, for every $\C V'\in\C P$, there is a hexagon containing the vertices in the set $\pi^{-1}(\C V')$. The diagram on the left (resp. right) of \Cref{fig:perfect match} shows the matching for a two-element (resp. three-element) subset $\C V'$ in $\C P$. Denote the subnetwork of $\C S[2]$ obtained by taking the union of all perfect matchings by $\C S'[2]$.

Let $\C S_1[2],\cdots,\C S_N[2]$ be all the pairwise distinct $\C R[2]$-systems relative to $\C V$. As discussed in the above paragraph, obtain the $\C R[2]$-free subnetworks $\C S'_1[2],\cdots,\C S'_N[2]$ satisfying $\C S'_k[2]_0=\C S_k[2]_0$ for each $k\in\{1,\cdots,N\}$. The subnetwork $\C M'$ defined as the union $\bigcup_{k=1}^N\C S'_k[2]$ along with the union of all the links and vertices of $\C M$ not present in $\bigcup_{k=1}^N\C S_k[2]$ is an $\C R[2]$-free subnetwork of $\C M$ such that $\C M'_0=\C M_0$.

We claim that $\C M'$ is complete. Indeed, by our construction of $\C M'$, for every link $\alpha$ in $\C M\setminus\C M'$, there exist links $\beta$ and $\gamma$ in $\C M'$ such that $\alpha\gamma\in\C R[2]$ and $\beta\alpha\in\C R[2]$. Therefore, the completeness of $\C M$ implies the completeness of $\C M'$.
\end{proof}

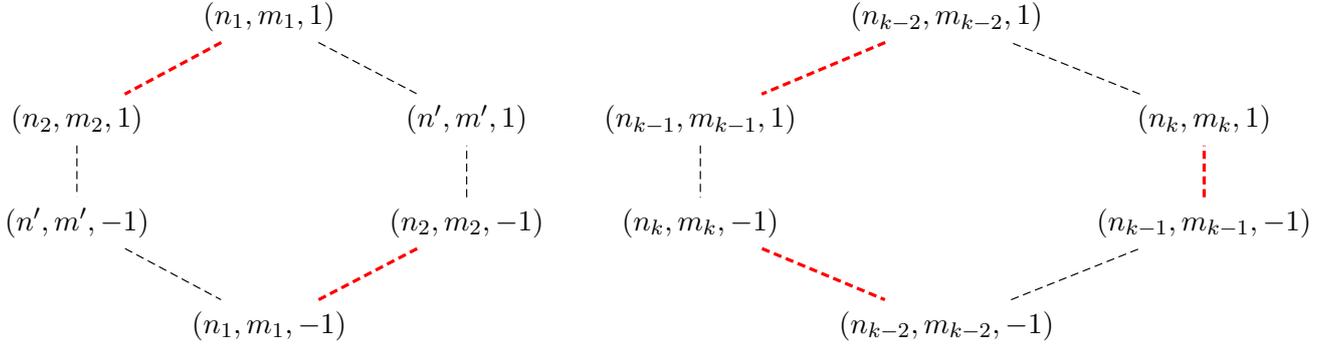
\begin{figure}[H]
    \[\begin{tikzcd}[column sep=tiny]
 & {(n_1,m_1,1)} \arrow[ld, no head, dashed, very thick, color=red] \arrow[rd, no head, dashed] &                                           &        &                                                  & {(n_{k-2},m_{k-2},1)} \arrow[ld, no head, dashed, very thick, color=red]  &                                                   \\
{(n_2,m_2,1)} \arrow[d, no head, dashed]   &                                                                       & {(n',m',1)}                             &  & {(n_{k-1},m_{k-1},1)} \arrow[d, no head, dashed] &                                                    & {(n_k,m_k,1)} \arrow[lu, no head, dashed]         \\
{(n',m',-1)} \arrow[rd, no head, dashed] &                                                                       & {(n_2,m_2,-1)} \arrow[u, no head, dashed] &        & {(n_k,m_k,-1)} \arrow[rd, no head, dashed, very thick, color=red]       &                                                    & {(n_{k-1},m_{k-1},-1)} \arrow[u, no head, dashed, very thick, color=red] \\
& {(n_1,m_1,-1)} \arrow[ru, no head, dashed, very thick, color=red]                            &                                           &        &                                                  & {(n_{k-2},m_{k-2},-1)} \arrow[ru, no head, dashed] &                                                  
\end{tikzcd}\]
    \caption{A perfect matching (in red) for an $\C R[2]$-system used in the proof of \Cref{prop: complete R free from complete involution invariant}}
    \label{fig:perfect match}
\end{figure}




\section{Generalised graph maps span the Hom-set}\label{sec: main}
Given a homomorphism $\C H:M_1\to M_2$, we aim at expressing it as a finite $\C K$-linear combination of homomorphisms $\C H_{\C G_i}$ corresponding to generalised graph maps $\C G_i$ (\Cref{thm: main 1}). Define the support of $\C H$, denoted $\R{supp}(\C H)$, as
$$\R{supp}(\C H):=\{(n,m)\in\C N[1]_0\ \mid\ \C H(v_n)=\sum_{m'\in T^2_0}\mu_{m'}w_{m'}\text{ with }\mu_m\neq0\}.$$
We will show that $\C M:=\ang{\pi^{-1}(\R{supp}\C H)}$ is complete (\Cref{prop: support is complete}), and then we apply \Cref{prop: complete R free from complete involution invariant} to obtain a complete $\C R[2]$-free subnetwork $\C M'$ of $\C M$. If $\C H$ is a non-zero homomorphism, then $\C M'$ is non-empty. In that case, a connected component $\C M''$ of $\C M'$ is not necessarily a generalised graph map as it may not be involution-free. If there exists $(n,m,j)\in\C M''_0$ for which $(n,m,-j)\notin\C M''_0$, then we will show in \Cref{lem} that there is an algorithm to carve out a generalised graph map $\C G$ from $\C M''$. For an appropriate $\mu\in\C K^\times$, we have $|\R{supp}(\C H-\mu\C{H_G})|<|\R{supp}(\C H)|$, which allows us to induct on $|\R{supp}(\C H)|$ to achieve the goal. 

However, if at some stage, all connected components $\C M''$ of a complete $\C R[2]$-free subnetwork $\C M'$ are \emph{involution-invariant}, i.e. $\C M''_0=\pi^{-1}(\pi(\C M''_0))$, then this algorithm fails to produce a generalised graph map.

\begin{defn}\label{defn: ghost}
Say that a subnetwork $\C M$ of $\C N[2]$ is a \emph{ghost} if $\C M$ is non-empty, connected, involution-invariant, complete and $\C R[2]$-free. Say that the pair $(M_1,M_2)$ of generalised tree modules is \emph{ghost-free} if the $2$-covering network $\C N[2]$ does not include any ghost.
\end{defn}

\begin{rmk}
The set of ghosts in the network $\C N[2]$ is finite.
\end{rmk}

\begin{rmk}\label{rmk: ghost give zero homo}
In view of \Cref{prop: complete and Rfree is homo}, if $\C M$ is a ghost, then $\C{H_M}$ is the zero homomorphism.
\end{rmk}

\begin{exmp}\label{exmp: ghost}
Consider the generalised tree modules $M_1$ and $M_2$ over $\C K\B A_2$, where the trees $T^1$ and $T^2$ are given in \Cref{fig: T1 T2 ghost} with $F_1(1)=F_1(2)=F_2(6)=F_2(8)=\B 1$, $F_1(4)=F_2(5)=F_2(7)=F_2(9)=\B 2$ and $F_1(a)=F_1(b)=F_2(d)=F_2(e)=F_2(f)=F_2(g)=\B a$. The pullback network $\C N[1]$ is shown in \Cref{fig: ghost N1} and the ghost in $\C N[2]$ is shown in \Cref{fig: ghost N2}.
\end{exmp}

\begin{figure}[H]
    \[\begin{tikzcd}
1 \arrow[rd, "a"'] & & 2 \arrow[ld, "b"] &   & 6 \arrow[ld, "d"'] \arrow[rd, "e"] &   & 8 \arrow[ld, "f"'] \arrow[rd, "g"] &   \\
                   & 4                 &                   & 5 &                                     & 7 &                                    & 9
\end{tikzcd}\]
    \caption{Trees $T^1$ and $T^2$ used in \Cref{exmp: ghost}}
    \label{fig: T1 T2 ghost}
\end{figure}
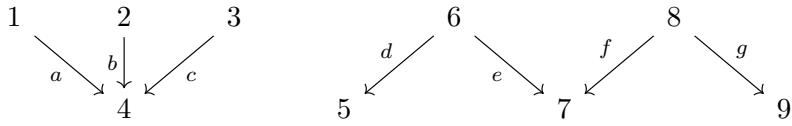

\begin{figure}[H]
    \[\begin{tikzcd}
        & {(1,6)} \arrow[rd, "{(a,e)}"] \arrow[rr, no head] \arrow[ld, "{(a,d)}"'] &         & {(1,8)} \arrow[ld, "{(a,f)}"'] \arrow[rd, "{(a,g)}"] &         \\
{(4,5)} &                                                                          & {(4,7)} &                                                      & {(4,9)} \\
        & {(2,6)} \arrow[ru, "{(b,e)}"'] \arrow[lu, "{(b,d)}"] \arrow[rr, no head] &         & {(2,8)} \arrow[ru, "{(b,g)}"'] \arrow[lu, "{(b,f)}"] &        
\end{tikzcd}\]
    \caption{The subnetwork $\C N[1]$ in \Cref{exmp: ghost} such that $\pi^{-1}(\C N[1])=\C N[2]$ contains a ghost}
    \label{fig: ghost N1}
\end{figure}
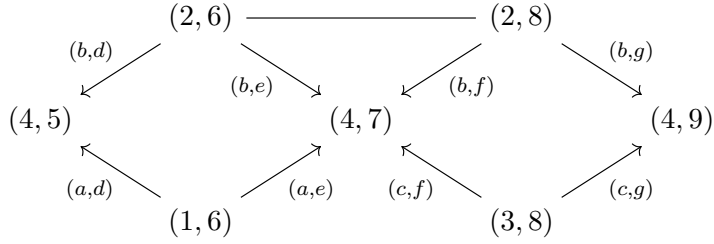

\begin{figure}[H]
    \[\begin{tikzcd}[row sep=34pt]
           & {(1,6,1)} \arrow[rd, "{(a,e,1)}"] \arrow[ld, "{(a,d,1)}"] \arrow[ddddd, no head, bend right=110, shift right=3,color=red] &            & {(1,8,1)} \arrow[ld, "{(a,f,1)}"',color=red] \arrow[rd, "{(a,g,1)}"'] \arrow[ddddd, no head, bend left=110, shift left=3] &            \\
{(4,5,1)}  &                                                                                                                & {(4,7,1)}  &                                                                                                                & {(4,9,1)}  \\
           & {(2,6,1)} \arrow[lu, "{(b,d,1)}"] \arrow[ru, "{(b,e,1)}"',color=red] \arrow[d, no head]                                  &            & {(2,8,1)} \arrow[lu, "{(b,f,1)}"] \arrow[ru, "{(b,g,1)}"'] \arrow[d, no head, color=red]                                  &            \\
           & {(2,8,-1)} \arrow[rd, "{(b,f,-1)}",color=red] \arrow[ld, "{(b,g,-1)}"']                                                  &            & {(2,6,-1)} \arrow[ld, "{(b,e,-1)}"'] \arrow[rd, "{(b,d,-1)}"]                                                  &            \\
{(4,9,-1)} &                                                                                                                & {(4,7,-1)} &                                                                                                                & {(4,5,-1)} \\
           & {(1,8,-1)} \arrow[lu, "{(a,g,-1)}"'] \arrow[ru, "{(a,f,-1)}"']                                                 &            & {(1,6,-1)} \arrow[lu, "{(a,e,-1)}",color=red] \arrow[ru, "{(a,d,-1)}"]&           
\end{tikzcd}\]
    \caption{A ghost (shown in black) in $\C N[2]$ used in \Cref{exmp: ghost}}
    \label{fig: ghost N2}
\end{figure}

We begin our journey to \Cref{thm: main 1} by proving an important property of the support of a homomorphism.

\begin{prop}\label{prop: support is complete}
Let $\C H:M_1\to M_2$ be a homomorphism. Then $\C M:=\ang{\pi^{-1}(\R{supp}(\C H))}$ is complete.
\end{prop}
\begin{proof}
Let $\C M=:(\C M_0,\C M_1,\C{E_M})$ and $(n,m,j)\in\C M_0=\R{supp}(\C H)\times\{-1,1\}$.\\

\noindent{}\textbf{Case 1.} Suppose $n'\in T^1_0$ and $(n'\xrightarrow{a}n)\in T^1_1$. Let $\gamma:=F_1(a)$ and $\C H(\gamma\cdot v_{n'})=\sum_{\underline{m}\in T^2_0}\mu_{\underline{m}}w_{\underline{m}}$.

Then exactly one of the following two cases holds.\\

\textbf{Subcase 1a.} $\mu_m\neq 0$.

Since $\C H$ is a homomorphism, we have $\C H(\gamma\cdot v_{n'})=\gamma\cdot\C H(v_{n'})$. This implies that $\C H(v_{n'})\neq 0$, which guarantees the existence of $m'\in T^2_0$ and $(m'\xrightarrow{b}m)\in T^2_1$ with $F_2(b)=\gamma$ and $(n',m')\in\R{supp}(\C H)$. As a result, we have $(n',m',j)\in\C M_0$ and $((n',m',j)\xrightarrow{(a,b,j)}(n,m,j))\in\C M_1$.\\

\textbf{Subcase 1b.} $\mu_m=0$.

Then there exist $n''\in T^1_0$ and $(n'\xrightarrow{b}n'')\in T^1_1$ with $F_1(b)=\gamma$ and $(n'',m)\in\R{supp}(\C H)$. As a result, we have $(n'',m,-j)\in\C M_0$ and $\{(n,m,j),(n'',m,-j)\}\in\C{E_M}$.\\

\noindent{}\textbf{Case 2.}  Suppose $m'\in T^1_0$ and $(m\xrightarrow{b}m')\in T^2_1$. Let $\beta:=F_2(b)$ and $\beta\cdot\C H(v_n)=\sum_{\underline{m}\in T^2_0}\mu_{\underline{m}}w_{\underline{m}}$.

Then exactly one of the following two cases holds.\\

\textbf{Subcase 2a.} $\mu_{m'}\neq 0$.

Since $\C H$ is a homomorphism, we have $\beta\cdot\C H(v_n)=\C H(\beta\cdot v_n)$. This implies that $\beta\cdot v_n\neq 0$, which guarantees the existence of $n'\in T^1_0$ and $(n\xrightarrow{a}n')\in T^1_1$ with $F_1(a)=\beta$ and $(n',m')\in\R{supp}(\C H)$. As a result, we have $(n',m',j)\in\C M_0$ and $((n,m,j)\xrightarrow{(a,b,j)}(n',m',j))\in\C M_1$.\\

\textbf{Subcase 2b.} $\mu_{m'}=0$.

Then there exist $m''\in T^2_0$ and $(m''\xrightarrow{a}m')\in T^2_1$ with $F_2(a)=\beta$ and $(n,m'')\in\R{supp}(\C H)$. As a result, we have $(n',m'',-j)\in\C M_0$ and $\{(n,m,j),(n,m'',-j)\}\in\C{E_M}$.
\end{proof}

Combining the above result with Proposition \ref{prop: complete R free from complete involution invariant}, we get an immediate consequence.

\begin{cor}\label{cor: complete R free from homo}
Let $\C H:M_1\to M_2$ be a homomorphism. Then there exists a complete $\C R[2]$-free subnetwork $\C M$ of $\C N[2]$ such that $\C M_0=\R{supp}(\C H)\times\{-1,1\}$.
\end{cor}

We wish to obtain a generalised graph map $\C G$ such that $\R{supp}(\C{H_G})\subseteq\R{supp}(\C H)$.

\begin{lem}\label{lem}
Suppose $(M_1,M_2)$ is ghost-free and $\C H:M_1\to M_2$ is a non-zero homomorphism. Then there exists a generalised graph map $\C G$ such that $\C G_0\subseteq\R{supp}(\C H)\times\{-1,1\}$.
\end{lem}
\begin{proof}
Let $\C M:=(\C M_0,\C M_1,\C{E_M})$ be a connected component of the subnetwork $\overline{\C M}$ obtained by applying Corollary \ref{cor: complete R free from homo} to $\C H$. Since $(M_1,M_2)$ is ghost-free, we have that $\C M_0$ is not involution-invariant. Therefore, there exists $(n_0,m_0,j_0)\in\C M'_0$ such that $(n_0,m_0,-j_0)\notin\C M'_0$.

Now, we will make some modifications to $\C M$ to get a subnetwork $\C M'$ of $\C N[2]$ with $\C M'_0=\C M_0$ satisfying the following property for every $(n,m,j),(n',m',j')\in\C M_0$:
\begin{align}\tag{$P$}
\text{if there is a link from $(n,m,j)$ to $(n',m',j')$ in $\C M'$ and $(n,m,-j)\in\C M'_0$,}\\\text{then $(n',m',-j')\in\C M'_0$, and there is a link from $(n,m,-j)$ to $(n',m',-j')$ in $\C M'$.}\notag 
\end{align}

If $\C M$ satisfies Property $(P)$, then choose $\C M':=\C M$. Otherwise, using the notations of the proof of \Cref{prop: complete R free from complete involution invariant} used in the construction of $\overline{\C M}$, there exists an $\C R[2]$-system $\C S[2]$ relative to $\R{supp}(\C H)$ such that
\begin{itemize}
    \item $|\C S[2]_0|=2k$ for odd $k$;
    \item $\pi^{-1}(\{(n_{k-2},m_{k-2},),\ (n_{k-1},m_{k-1}),\ (n_k,m_k)\})\subseteq\C S[2]_0$ such that three links have been chosen in the hexagon spanned by the above-mentioned six vertices as shown in \Cref{fig:perfect match} in the process of obtaining a perfect matching; and
    \item at least two chosen links in such a hexagon appear in $\C M$.
\end{itemize}
Note that completeness ensures that $|\pi^{-1}(\{(n_{k-2},m_{k-2},),\ (n_{k-1},m_{k-1}),\ (n_k,m_k)\})\cap\C M_0|\in\{4,6\}$. Make the following modifications in the choice of links for such hexagons in each of such $\C R[2]$-system based on the following two cases to get $\C M'$.

\textbf{Case 1.} Only two chosen links in the hexagon appear in $\C M$.

Suppose the chosen links in $\C M$ for such a hexagon are the links from $(n,m,j)$ to $(n',m',j')$ and that from $(n,m,-j)$ to $(n'',m'',j'')$. Then remove these links and add the link from $(n',m',j')$ to $(n'',m'',j'')$.\\

\textbf{Case 2.} All three chosen links appear in $\C M$.

In this case, with the existing enumeration of $\C S[2]_0$, replace the existing links with the links in the subnetwork $\ang{\pi^{-1}\{(n_{k-1},m_{k-1}),(n_{k-2},m_{k-2})\}}$ of $\C N[2]$.

It is obvious that the subnetwork $\C M'$ of $\C N[2]$ is $\C R[2]$-free and satisfies Property $(P)$; however, it is not necessarily complete. Let $\C G:=(\C G_0,\C G_1,\C{E_G})$ be a connected component of  $\C M'$ containing $(n_0,m_0,j_0)$.\\

\noindent\textbf{Claim:} If $(n,m,1),(n,m,-1)\in\C M_0$ then $(n,m,j)\notin\C G_0$ for any $j\in\{-1,1\}$.\\
\noindent\textit{Proof of the claim:} Suppose not. Then, since $\C G$ is connected, there is a traversal from $(n,m,1)$ to $(n_0,m_0,j_0)$ in $\C M'$. Using Property $(P)$ repeatedly with $(n,m,-1)\in\C G_0$, we obtain $(n_0,m_0,-j_0)\in\C M'_0$, a contradiction to the choice of $(n_0,m_0,j_0)$, thus proving the claim.\\

We now show that $\C G$ is a generalised graph map.\\

\noindent\textbf{$\C G_0$ is involution-free}: Same argument as in the proof of the claim above.\\

\noindent\textbf{$\C G$ is complete}: Note that all the links from a vertex of $\C G_0$ to a vertex in $\C M_0\setminus\C G_0$ in $\C M$ have been replaced by a link between two vertices of $\C G_0$, and any link between two vertices in $\C G_0$ present in $\C M$ is already present in $\C G$. Therefore, in view of the above two points, the completeness of $\C M$ ensures the completeness of $\C G$.
\end{proof}

Now we are ready to state and prove the main result of this paper.
\mainone*
\begin{proof}
The proof is by induction on $k:=|\R{supp}(\C H)|$.

\textbf{Base Case.} $k=0$.

Then $\C H$ is the zero homomorphism.\\

\textbf{Induction Step.} $k>0$.

Assume the conclusion holds for any homomorphism $\C H'$ with $|\mathrm{supp}(\C H')|<k$.

By \Cref{lem}, there exists a generalised graph map $\C G$ such that $\R{supp}(\C G)\subseteq\R{supp}(\C H)$. For $(n,m)\in\R{supp}(\C G)$, let $\C H(v_n)=\sum_{\underline{m}\in T^2_0}\mu_{\underline{m}}w_{\underline{m}}$ so that $\mu_m\neq0$. For exactly one $j\in\{-1,1\}$, we have $(n,m,j)\in\C G_0$. Then the homomorphism $\C H':=\C H-j\mu_m\C{H_G}$ satisfies $(n,m)\notin\R{supp}(\C H')$ and $\R{supp}(\C H')\subseteq\R{supp}(\C H)$, which together imply $|\R{supp}(\C H')|<|\R{supp}(\C H)|=k$. Applying the induction hypothesis on $\C H'$, we get the desired conclusion for $\C H$.
\end{proof}

\begin{que}\label{ques: main1}
Does the conclusion of \Cref{thm: main 1} remain true when the ghost-free hypothesis on the pair $(M_1,M_2)$ is dropped?
\end{que}

We believe that the answer to the above question is in the affirmative.

\section{Indecomposability of some generalised tree modules}\label{sec: indecomposability}
Let $T:=T^1=T^2$ and $F:=F_1=F_2$ so that $M:=F_\lambda(V_T)$ is a generalised tree module. The basis elements of $M$ will be denoted $\{v_n\}_{n\in T_0}$. Using \Cref{thm: main 1}, we provide a checkable sufficient condition (\Cref{thm: main 2}) for the indecomposability of $M$. In the direction of the converse of \Cref{thm: main 2}, a sufficient condition for some generalised tree modules to be decomposable is stated in \Cref{thm: conv main2}.

\maintwo*

If $\C H:M\to M$ is an endomorphism, then we expand $\C H(v_n)$ in terms of the standard basis vectors as $$\C H(v_n):=\sum_{(n,m)\in T_0\times T_0}\mu^{\C H}_{nm}v_m\text{ for }n\in T_0.$$ 

Before providing a proof of \Cref{thm: main 2}, we provide as consequences of \Cref{thm: main 1} two statements imposing conditions on the coefficients of the form $\mu^{\C H}_{nm}$ that are equivalent to Conditions (1) and (2) of the hypotheses of \Cref{thm: main 2}.

\begin{rmk}\label{rmk: hyp1 thB equiv}
Under the ghost-free hypothesis on the pair $(M,M)$, Condition (1) of the hypotheses of \Cref{thm: main 2} is equivalent to the following statement: for every subtree of $T$ of the form $n_1\xleftarrow{a}n\xrightarrow{b}n_2$ or $n_1\xrightarrow{a}n\xleftarrow{b}n_2$ with $n_1\neq n_2$ and satisfying $F(a)=F(b)$, and for any endomorphism $\C H:M\to M$, we have $\mu^{\C H}_{n_1n_2}=0$.
\end{rmk}

\begin{rmk}\label{rmk: hyp2 thB equiv}
Under the ghost-free hypothesis on the pair $(M,M)$, Condition (2) of the hypotheses of \Cref{thm: main 2} is equivalent to the following statement: for $n_1\neq n_2$ in $T_0$, and for any pair of endomorphisms $\C H,\C H':M\to M$ we have $\mu^{\C H}_{n_1n_2}\mu^{\C H'}_{n_2n_1}=0$.
\end{rmk}

\begin{proof}[Proof of Theorem B]
Let $\C I\in\R{End}_\Lambda(M)$ be a non-zero idempotent endomorphism of $M$. Choose $n_0\in T_0$. Without loss of generality, assume that $(n_0,n_0)\in\R{supp}(\C I)$ (for otherwise, $(n_0,n_0)$ lies in the support of the non-zero idempotent endomorphism $(\B 1_M-\C I)$, and we can work with $(\B 1_M-\C I)$ instead). In other words, $\mu^{\C I}_{n_0n_0}\neq0$. Then $\mu^{\C I}_{n_0n_0}=\mu^{\C I^2}_{n_0n_0}=\sum_{m\in T_0}\mu^{\C I}_{n_0m}\mu^{\C I}_{mn_0}=(\mu^{\C I}_{n_0n_0})^2$, where the third equality follows from \Cref{rmk: hyp2 thB equiv}. Since $\C K$ is a field, we conclude $\mu^{\C I}_{n_0n_0}=1$.

We will now show that $\mu^{\C I}_{nn}=1$ for every $n\in T_0$. The proof follows by induction on the tree $T$.\\

\textbf{Base Case.} We already have $\mu^{\C I}_{n_0n_0}=1$.\\

\textbf{Inductive Step.} Suppose we have $\mu^{\C I}_{nn}=1$. There are two cases to proceed to vertices adjacent to $n$ in $T$.\\

\textbf{Case 1.} There exists $(n'\xrightarrow{a}n)\in T_1$.

Let $\gamma:=F(a)$. Since $\C I$ is a homomorphism, we have $\gamma\cdot\C I(v_{n'})=\C I(\gamma\cdot v_{n'})$.

Suppose $\gamma\cdot v_{n'}=v_n+\sum_{i=1}^kv_{n_i}$ and $\C I(\gamma\cdot v_{n'})=\sum_{m\in T_0}\mu_mv_m$. Then $\mu_n=1+\sum_{i=1}^k \mu_{n_in}^{\C I}$. If $\mu_n\neq1$, then there exists $(n'\xrightarrow{b}n_i)\in T_1$ with $F(b)=\gamma$ and $\mu^{\C I}_{n_in}\neq0$ for some $i\in\{1,\cdots,k\}$, which contradicts Condition (1) of the hypotheses of the theorem in view of \Cref{rmk: hyp1 thB equiv}. Thus, $\mu_n=1$, and hence $\C I(v_{n'})\neq 0$.

Similarly, if $\mu^{\C I}_{n'n'}\neq1$, there exists $n''(\neq n')\in T_0$ and $(n''\xrightarrow{c}n)\in T_1$ with $F(c)=\gamma$ such that $\mu^{\C I}_{n'n''}\neq0$, a contradiction to Condition (1) of the hypotheses of the theorem in view of \Cref{rmk: hyp1 thB equiv}. Therefore, $\mu^{\C I}_{n'n'}=1$.\\

\textbf{Case 2.} There exists $(n\xrightarrow{a}n')\in T_1$.

Let $\gamma:=F(a)$. Since $\C I$ is a homomorphism, we have $\gamma\cdot\C I(v_{n})=\C I(\gamma\cdot v_{n})$.

Suppose $\gamma^{-1}(n'):=\{m\in T_0\mid (m\xrightarrow{b}n')\in T_1, F(b)=\gamma\}$ and $\gamma\cdot\C I(v_n)=\sum_{m\in T_0}\mu_mv_m$. Then $\mu_{n'}=1+\sum_{n''\in\gamma^{-1}(n')}\mu^{\C I}_{nn''}$. If $\mu_{n'}\neq1$, then there exists $n''\in\gamma^{-1}(n')$ such that  $\mu^{\C I}_{nn''}\neq0$. By definition of $\gamma^{-1}(n')$, $(n''\xrightarrow{b}n')\in T_1$ satisfies $F(b)=\gamma$, which contradicts Condition (1) of the hypotheses of the theorem in view of \Cref{rmk: hyp1 thB equiv}. Thus $\mu_{n'}=1$, and hence $\gamma\cdot v_n\neq0$.

Similarly, if $\mu^{\C I}_{n'n'}\neq1$, then there exists $n'''(\neq n')\in T_0$ and $(n\xrightarrow{c}n''')\in T_1$ with $F(c)=\gamma$ such that $\mu^{\C I}_{n'''n'}\neq0$, a contradiction to the Condition (1) of the hypotheses of the theorem in view of \Cref{rmk: hyp1 thB equiv}. Therefore, $\mu_{n'n'}^{\C I}=1$.\\

Therefore, we have proved that $\mu^{\C I}_{nn}=1$ for every $n\in T_0$. Since the only possible eigenvalues of an idempotent matrix are $0$ or $1$, and the trace of an idempotent matrix equals its rank, we conclude that $\C I$ is a full-rank idempotent linear map, and hence invertible. It then immediately follows that $\C I=\B 1_M$. Then \Cref{prop: indecomposability condn assem} yields the indecomposability of $M$.
\end{proof}

Here is an immediate application of the theorem above. Suppose $M:=F_{\lambda}(V_T)$ is a tree module. Recall the definition of a word formed by arrows of the quiver $Q$ and their inverses from the paragraph before \Cref{rmk: traversal in N1 has a unique string}. For every ordered pair $(n,n')$ of vertices of $T$, let $\s w(n,n')$ be the word defined by the $F$-image of the shortest zigzag in $T$ from $n$ to $n'$--the tree condition ensures that $\s w(n,n')$ is reduced. Clearly $\s w(n,n')^{-1}=\s w(n',n)$ and $\s w(n,n)$ is the lazy path at $F(n)$. Moreover, given a reduced word $\s w$ and $n\in T_0$ (resp. $n'\in T_0$), there exists at most one $n'\in T_0$ (resp. $n\in T_0$) such that $\s w=\s w(n,n')$.

\begin{cor}\label{cor: indecomposable gtm - 1 tree}
If $M:=F_{\lambda}(V_T)$ is a tree module, then $M$ is indecomposable.
\end{cor}
\begin{proof}
The non-existence of subtrees of the form $n_1\xleftarrow{a}n\xrightarrow{b}n_2$ or $n_1\xrightarrow{a}n\xleftarrow{b}n_2$ in $T$ satisfying $F(a)=F(b)$ implies that Condition (1) of the hypotheses of \Cref{thm: main 2} is vacuously satisfied for $M$. Consequently, the non-existence of edges in $\C N[1]$, and hence in $\C N[2]$, implies the non-existence of a traversal from $(n,m,j)$ to $(n,m,-j)$ for any $(n,m,j)\in\C N[2]$. Therefore, the pair $(M,M)$ is ghost-free.

We will now show that Condition (2) of the hypotheses of \Cref{thm: main 2} holds for $M$. If possible, let $n_1\neq n_2\in T_0$ be such that there exist generalised graph maps $\C G$ and $\C G'$ with $(n_1,n_2)\in\pi(\C G_0)$ and $(n_2,n_1)\in\pi(\C G'_0)$. Without loss of generality, assume that the length of the path between $n_1$ and $n_2$ in the underlying undirected graph of $T$ is minimal among such pairs. Starting with the pair $(n_1,n_2)$ and applying Conditions (1a) and (2a) of \Cref{defn: complete subnetwork} repeatedly to both $\C G$ and $\C G'$, we get a quiver automorphism $\nu$ of $T$ taking $n_1$ to $n_2$. Let $\s w:=\s w(n_1,n_2)$.

Suppose $\s w=\alpha_n\cdots\alpha_1$. The minimality of the length of the path between $n_1$ and $n_2$ in the underlying graph of $T$ forces that $\alpha_n\neq\alpha_1^{-1}$ so that $\s w^2$ (or equivalently, $\s w^m$ for any $m\in\mathbb N^+$) is reduced. As a result, for $l<k$ we have $\s w(\nu^l(n_1),\nu^k(n_1))=\s w^{k-l}$. Given finiteness of $T_0$, there exist $l<k$ such that $\nu^k(n_1)=\nu^l(n_1)$. But then $\s w^{k-l}=\s w(\nu^k(n_1),\nu^k(n_1))$, and the latter is the lazy path at $F(\nu^k(n_1))$. This implies that $\s w$ itself is the lazy path at $F(n_1)$, and hence $n_1=n_2$, a contradiction. The conclusion then follows from \Cref{thm: main 2}.
\end{proof}

We defer the applications of this result to prove indecomposability of certain generalised tree modules to \S~\ref{sec:Dn indecomposable}.

We have the following result in the direction of the converse of \Cref{thm: main 2}.

\begin{thm}\label{thm: conv main2}
(A) Suppose $T$ is a tree as shown in \Cref{fig: conv main2}. Denote by $\overline T^j$ the induced tree $\ang{T^j_0\cup\{0\}}$ for each $1\leq j\leq k$. Further suppose $F:T\to( Q,\rho)$ is a bound quiver morphism satisfying the following conditions:
\begin{enumerate}
    \item \label{CB condition for Tj} $F(c)\neq F(d)$ whenever $c,d\in\overline T^j_1$ with $c\neq d$, and $s(c)=s(d)$ or $t(c)=t(d)$ for every $1\leq j\leq k$;
    \item \label{same ai} $F(a_1)=F(a_j)$ for all $1\leq j\leq k$, and if $b\in T_1$ with $s(b)=0$ and $F(b)=F(a_1)$ then $b=a_j$ for some $1\leq j\leq k$.
\end{enumerate}
If there exist distinct $n_1,n_2\in\{1,\cdots,k\}$ and a generalised graph map $\C G$ with $(n_1,n_2)\in\pi(\C G_0)$, then $M$ is decomposable.

(B) Suppose $T$ is a tree as shown in \Cref{fig: conv main2} but with the direction of $a_j$ reversed for every $j\in\{1,\cdots,k\}$. Denote by $\overline T^j$ the induced tree $\ang{T^j_0\cup\{0\}}$ for each $1\leq j\leq k$. Suppose $F:T\to( Q,\rho)$ is a bound quiver morphism satisfying the following conditions:
\begin{enumerate}
    \item \label{CB condition for Tj2} $F(c)\neq F(d)$ whenever $c,d\in\overline T^j_1$ with $c\neq d$, and $s(c)=s(d)$ or $t(c)=t(d)$ for every $1\leq j\leq k$;
    \item \label{same ai2} $F(a_1)=F(a_j)$ for all $1\leq j\leq k$, and if $b\in T_1$ with $t(b)=0$ and $F(b)=F(a_1)$ then $b=a_j$ for some $1\leq j\leq k$.
\end{enumerate}
If there exist distinct $n_1,n_2\in\{1,\cdots,k\}$ and a generalised graph map $\C G$ with $(n_1,n_2)\in\pi(\C G_0)$, then $M$ is decomposable.
\end{thm}
\begin{figure}[H]
\begin{center}
\begin{tikzpicture}[x=0.7pt,y=0.5pt,yscale=-1,xscale=1]

\draw    (310,206) -- (211.64,274.85) ;
\draw [shift={(210,276)}, rotate = 325.01] [color={rgb, 255:red, 0; green, 0; blue, 0 }  ][line width=0.75]    (10.93,-3.29) .. controls (6.95,-1.4) and (3.31,-0.3) .. (0,0) .. controls (3.31,0.3) and (6.95,1.4) .. (10.93,3.29)   ;
\draw    (340,206) -- (435.39,276.81) ;
\draw [shift={(437,278)}, rotate = 216.59] [color={rgb, 255:red, 0; green, 0; blue, 0 }  ][line width=0.75]    (10.93,-3.29) .. controls (6.95,-1.4) and (3.31,-0.3) .. (0,0) .. controls (3.31,0.3) and (6.95,1.4) .. (10.93,3.29)   ;
\draw    (320,206) -- (296.66,273.11) ;
\draw [shift={(296,275)}, rotate = 289.18] [color={rgb, 255:red, 0; green, 0; blue, 0 }  ][line width=0.75]    (10.93,-3.29) .. controls (6.95,-1.4) and (3.31,-0.3) .. (0,0) .. controls (3.31,0.3) and (6.95,1.4) .. (10.93,3.29)   ;
\draw  [dash pattern={on 4.5pt off 4.5pt}]  (330,206) -- (342.64,276.03) ;
\draw [shift={(343,278)}, rotate = 259.77] [color={rgb, 255:red, 0; green, 0; blue, 0 }  ][line width=0.75]    (10.93,-3.29) .. controls (6.95,-1.4) and (3.31,-0.3) .. (0,0) .. controls (3.31,0.3) and (6.95,1.4) .. (10.93,3.29)   ;
\draw   (258,338) .. controls (258,304.86) and (273.67,278) .. (293,278) .. controls (312.33,278) and (328,304.86) .. (328,338) .. controls (328,371.14) and (312.33,398) .. (293,398) .. controls (273.67,398) and (258,371.14) .. (258,338) -- cycle ;
\draw   (160,336) .. controls (160,302.86) and (175.67,276) .. (195,276) .. controls (214.33,276) and (230,302.86) .. (230,336) .. controls (230,369.14) and (214.33,396) .. (195,396) .. controls (175.67,396) and (160,369.14) .. (160,336) -- cycle ;
\draw   (420,336) .. controls (420,302.86) and (435.67,276) .. (455,276) .. controls (474.33,276) and (490,302.86) .. (490,336) .. controls (490,369.14) and (474.33,396) .. (455,396) .. controls (435.67,396) and (420,369.14) .. (420,336) -- cycle ;
\draw   (285,143) .. controls (285,109.86) and (300.67,83) .. (320,83) .. controls (339.33,83) and (355,109.86) .. (355,143) .. controls (355,176.14) and (339.33,203) .. (320,203) .. controls (300.67,203) and (285,176.14) .. (285,143) -- cycle ;

\draw (315,183.4) node [anchor=north west][inner sep=0.75pt]    {$0$};
\draw (197,278.4) node [anchor=north west][inner sep=0.75pt]    {$1$};
\draw (442,279.4) node [anchor=north west][inner sep=0.75pt]    {$k$};
\draw (287,278.4) node [anchor=north west][inner sep=0.75pt]    {$2$};
\draw (330,278.4) node [anchor=north west][inner sep=0.75pt]    {$\cdots $};
\draw (181,398.4) node [anchor=north west][inner sep=0.75pt]    {${T}^{1} \ $};
\draw (283,398.4) node [anchor=north west][inner sep=0.75pt]    {${T}^{2} \ $};
\draw (443,398.4) node [anchor=north west][inner sep=0.75pt]    {${T}^{k} \ $};
\draw (258,212.4) node [anchor=north west][inner sep=0.75pt]    {$a_{1}$};
\draw (288,232.4) node [anchor=north west][inner sep=0.75pt]    {$a_{2}$};
\draw (375,213.4) node [anchor=north west][inner sep=0.75pt]    {$a_{k}$};
\draw (315,60.4) node [anchor=north west][inner sep=0.75pt]    {$T^{0} \ $};
\end{tikzpicture}
\end{center}
    \caption{Tree $T$ used in \Cref{thm: conv main2}(A)}
    \label{fig: conv main2}
\end{figure}
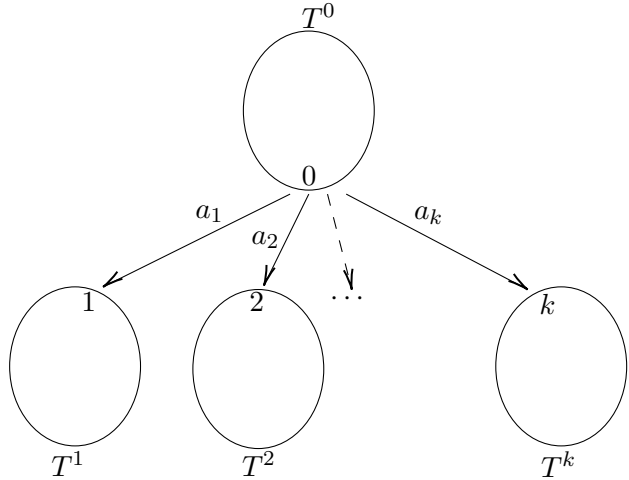
\begin{proof}
(A) In view of \Cref{prop: indecomposability condn assem}, it is sufficient to produce a non-trivial idempotent homomorphism $\C I:M\to M$. By the hypotheses, we have a generalised graph map $\C G$ and distinct $n_1,n_2\in\{1,\cdots,k\}$ such that $(n_1,n_2)\in\pi(\C G_0)$. In view of \Cref{rmk: -generalised graph map}, without loss of generality, let $(n_1,n_2,1)\in\C G_0$. Let $\C N'[2]$ be the 2-covering network associated with the pair $(M'_1:=F_\lambda(V_{T^{n_1}}),M'_2:=F_\lambda(V_{T^{n_2}}))$, which is a subnetwork of $\C N[2]$ associated with the pair $(M,M)$. Note that Condition (1) for the quiver morphism $F$ implies that $M'_1$ and $M'_2$ are tree modules.\\

\noindent\textbf{Claim:} The subnetwork $\C G\cap\C N'[2]$ is a graph map from $M'_1$ to $M'_2$.\\
\textit{Proof of the claim.}
Recall from the beginning of \S~\ref{sec: ggm} an alternate description of a graph map as a subnetwork of $\C N[2]$. The non-existence of an incoming arrow $a(\neq a_1)$ at $n_1$ with $F(a)=F(a_1)$, thanks to Condition (1) for the quiver morphism $F$, guarantees the existence of a unique $m'$ (resp. $n'$) for every $(n'\xrightarrow{a}n)$ in $T_1^{n_1}$ (resp. for every $(m\xrightarrow{b}m')$ in $T_1^{n_2}$) such that $(n',m',1)\in\C G_0$ starting at $(n_1,n_2,1)\in\C G_0$ thereby proving the claim.

Consider a subnetwork $\C G':=(\C G'_0,\C G'_1,\C E_{\C G'})$ of $\C N[2]$ defined as follows:
\begin{align*}
    \C G'_0&:=(\C G\cap\C N'[2])_0\cup\{(n,n,1)\in\C N[2]_0\mid n\in T_0\setminus{T}_0^{n_2}\};\\
    \C G'_1&:=(\C G\cap\C N'[2])_1\cup\{(a,a,1)\in\C N[2]_1\mid a\in T_1\setminus\overline T^{n_2}_1\}\cup\{(a_{n_1},a_{n_2},1)\};\text{ and}\\
    \C{E_{G'}}&:=\emptyset.
\end{align*}

We will now show that $\C G'$ is a generalised graph map from $M$ to $M$.\\

\noindent\textbf{$\C G'$ is involution-free}: Since $\C G\cap\C N'[2]$ is a graph map as shown in the claim, and $(n_1,n_2,1)\in(\C G\cap\C N'[2])_0$, we have that $(n,m,-1)\notin(\C G\cap\C N'[2])_0$ for any $(n,m)\in\C N[1]$. Therefore, we have $(n,m,-1)\notin\C G'_0$ for any $(n,m)\in\C N[1]$, which implies that $\C G'_0$ is involution-free.\\

\noindent\textbf{$\C G'$ is complete and $\C R[2]$-free:} The completeness conditions at $(n,n,1)\in\C N[2]_0$, $n\in T_0\setminus T^{n_2}_0$, for any $a\in T_1\setminus\overline T^{n_2}_1$ are uniquely satisfied by $(a,a,1)\in\C G'_1$; at $(n,m,1)\in(\C G\cap\C N'[2])_0$ for $a\in T_1\setminus\{a_{n_1}\}$ with source $m$ or target $n$ by the unique choices provided by the graph map $\C G\cap\C N'[2]$; at $(0,0,1)$ for $(0\xrightarrow{a_{n_2}}n_2)$ by $(a_{n_1},a_{n_2},1)\in\C G'_1$ since $(n_1,n_2,1)\in\C G'_0$; and finally at $(n_1,n_2,1)$ for $(0\xrightarrow{a_{n_1}}n_1)$ again by $(a_{n_1},a_{n_2},1)\in\C G'_1$. Since all the choices are unique, in view of \Cref{prop: uniqueness=Rfree}, we conclude that $\C G'$ is $\C R[2]$-free too.\\

\noindent\textbf{$\C G'$ is connected:} Since $T$ is a tree, there is a traversal in $\C G'$ from $(n,n,1)$ to $(0,0,1)$ for all $(n,n,1)\in\C G'_0$. Moreover, there is the link $(0,0,1)\xrightarrow{(a_{n_1},a_{n_2},1)}(n_1,n_2,1)$. Hence, the connectedness of the subnetwork $\C G\cap\C N'[2]$ implies the connectedness of $\C G'$.\\

We will now show that the homomorphism $\C I:=\C H_{\C G'}$ (cf. \Cref{prop: complete and Rfree is homo}) is a non-trivial idempotent homomorphism. It is clearly non-trivial as $\C I(v_{n_2})=0$ and $\C I(v_0)=v_0$. Now, it is sufficient to check for idempotency at those basis elements $v_n$ for which $\C I(v_n)\notin\{0,v_n\}$. By the construction, each such $n$ lies in ${T}^{n_1}_0$, and we have $\C I(v_n)=v_n+v_{n'}$ for some $n'\in{T}_0^{n_2}$. Since $\C I(v_{n'})=0$, we get $\C I^2(v_n)=\C I(v_n)$ as required, thereby completing the proof.

(B) In this case, we choose the generalised graph map $\C G':=(\C G'_0,\C G'_1,\C{E_{G'}})$ as follows:
\begin{align*}
    \C G'_0&:=(\C G\cap\C N'[2])_0\cup\{(n,n,1)\in\C N[2]_0\mid n\in T_0\setminus{T}_0^{n_1}\};\\
    \C G'_1&:=(\C G\cap\C N'[2])_1\cup\{(a,a,1)\in\C N[2]_1\mid a\in T_1\setminus\overline T^{n_1}_1\}\cup\{(a_{n_1},a_{n_2},1)\};\text{ and}\\
    \C{E_{G'}}&:=\emptyset.
\end{align*}
The rest of the proof is similar to that of (A) above.
\end{proof}

Once again, we defer the applications of the above theorem to \S~\ref{sec:Dn indecomposable}.

We expect that the conclusion of \Cref{thm: conv main2} holds for an arbitrary generalised tree module.

\begin{conj}\label{conj}
Suppose $M:=F_\lambda(V_T)$ is a generalised tree module. If there exist a subtree $T_1$ of $T$ of the form $n_1\xleftarrow{a}n\xrightarrow{b}n_2$ or $n_1\xrightarrow{a}n\xleftarrow{b}n_2$, and a generalised graph map $\C G$ with  $F(a)=F(b)$ and $(n_1,n_2)\in\pi(\C G_0)$, then $M$ is decomposable.
\end{conj}

\section{All indecomposables for Dynkin quivers of type $\B D$ are generalised tree modules}\label{sec:Dn indecomposable}

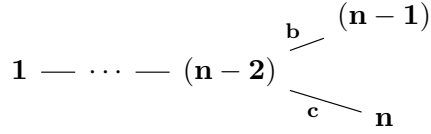
\begin{figure}[H]
    \[\begin{tikzcd}[column sep=small, row sep=2pt]
                         &                           &                                                                    & \B{(n-1)} \\
\B{1} \arrow[r, no head] & \cdots \arrow[r, no head] & \B{(n-2)} \arrow[ru, "\B b", no head] \arrow[rd, "\B c"', no head] &           \\
                         &                           &                                                                    & \B{n}    
\end{tikzcd}\]
    \caption{Dynkin diagram $\B D_n$}
    \label{fig:DN quiver}
\end{figure}

Consider a Dynkin quiver $\B{Q}=(\B{Q}_0,\B{Q}_1,\B s,\B t)$ of type $\B{D}_n$ (see \Cref{fig:DN quiver}). Let $\B b,\B c$ be the unique arrows of the quiver $\B Q$ satisfying $\{\B s(\B b),\B t(\B b)\}=\{\B n-\B 2,\B n-\B 1\}$ and $\{\B s(\B c),\B t(\B c)\}=\{\B n-\B 2,\B n\}$.  The indecomposable representations over $\B Q$ were classified by Gabriel \cite{gabriel1972unzerlegbare} in terms of positive roots of the corresponding root systems (or equivalently, the dimension vectors of such representations). We refer an interested reader to \cite[\S~8.3.2]{schiffler2014quiver} and \cite[\S~VII.2]{assem2006elements} for more discussion on Dynkin quivers of type $\B D$ as well as for the list of such dimension vectors. For each prescribed dimension vector from a subset of the list, we construct three generalised tree modules. We show using Theorems \ref{thm: main 2} and \ref{thm: conv main2} that the ``correct choices" of generalised tree modules constructed this way are indecomposable while the ``wrong choices'' lead to decomposability. As discussed towards the end of \S~\ref{sec: intro}, since all indecomposable $\C K\B Q$-modules are exceptional (see \cite[(8)]{ringel2006book}), a result of Ringel \cite{ringelexceptional} also yields that they are isomorphic to generalised tree modules.

An arbitrary dimension vector of an indecomposable module $M=((U_v)_{v\in\B Q_0},(\psi_a)_{a\in\B Q_1})$ over $\B Q$ has $\R{dim}(U_v)\in\{0,1,2\}$. If $\R{dim}(U_v)\leq1$ for every $v\in\B Q_0$, then $M$ is a tree module.

Now assume that $M$ is not a tree module. Then $\overline{\dim}(M)$ is of the form $0\cdots01\cdots12\cdots2{}^1_1$, where the strings of $1$s and $2$s have non-zero lengths. Let $\B a\notin\{\B b,\B c\}$ be the unique arrow of the quiver $\B Q$ satisfying $\{\R{dim}(U_{\B{s(a)}}),\R{dim}(U_{\B{t(a)}})\}=\{1,2\}$. Using the pigeonhole principle, there exist $\B a',\B b'\in\{\B a,\B b,\B c\}$ with $\B a'\neq\B b'$ and $\R{dim}(U_{\B t(\B a')})=\R{dim}(U_{\B t(\B b')})$. Let $\B c'$ be the unique element in $\{\B{a,b,c}\}\setminus\{\B{a',b'}\}$. Consider the generalised tree modules $M:=F_\lambda(V_T)$, where $T$ is shown in \Cref{fig:ind modules of DN} based on the choice of $\B c'$--here the quiver morphism $F$ is clear from the diagrams in each case.

\begin{figure}[H]
\vspace{7mm}
\begin{subfigure}{0.3\textwidth}
\[\begin{tikzcd}[column sep=small, row sep=2pt,
every matrix/.append style = {name=m},remember picture, overlay]
&&&& \bullet \\
&& \bullet \arrow[r, no head, dashed] & \bullet \arrow[ru, "b", no head]&\\
{\bullet}\arrow[r, no head, dashed] & \bullet \arrow[rd, "a_2"', no head] \arrow[ru, "a_1", no head] &&&\\
&& \bullet \arrow[r, no head, dashed] & \bullet \arrow[rd, "c"', no head] &\\
&&&& \bullet
\end{tikzcd}\]
\begin{tikzpicture}[remember picture, overlay,E/.style = {rectangle, draw=red, inner xsep=0.8pt,inner ysep=0.8pt, fit=#1}]
\node[E = (m-2-3) (m-4-3)] {};
\node[E = (m-2-4) (m-4-4)] {};
\end{tikzpicture}
\vspace{9mm}
\caption{$\B c'=\B a$}
\end{subfigure}
\begin{subfigure}{0.3\textwidth}
\[\begin{tikzcd}[column sep=small, row sep=2pt,
every matrix/.append style = {name=m},remember picture, overlay]
                                   &                                  &                                    &                                    & \bullet \arrow[lddd, "b_2", no head] \\
                                   &                                  & \bullet \arrow[r, no head, dashed] & \bullet \arrow[ru, "b_1", no head] &                                      \\
\bullet \arrow[r, no head, dashed] & \bullet \arrow[ru, "a", no head] &                                    &                                    &                                      \\
                                   &                                  & \bullet \arrow[r, no head, dashed] & \bullet \arrow[rd, "c"', no head]  &                                      \\
                                   &                                  &                                    &                                    & \bullet                             
\end{tikzcd}\]
\begin{tikzpicture}[remember picture, overlay,E/.style = {rectangle, draw=red, inner xsep=0.8pt,inner ysep=0.8pt, fit=#1}]
\node[E = (m-2-3) (m-4-3)] {};
\node[E = (m-2-4) (m-4-4)] {};
\end{tikzpicture}
\vspace{9mm}
\caption{$\B c'=\B b$}
\end{subfigure}
\begin{subfigure}{0.3\textwidth}
\[\begin{tikzcd}[column sep=small, row sep=2pt,
every matrix/.append style = {name=m},remember picture, overlay]
                                   &                                   &                                    &                                     & \bullet                               \\
                                   &                                   & \bullet \arrow[r, no head, dashed] & \bullet \arrow[ru, "b", no head]    &                                       \\
\bullet \arrow[r, no head, dashed] & \bullet \arrow[rd, "a"', no head] &                                    &                                     &                                       \\
                                   &                                   & \bullet \arrow[r, no head, dashed] & \bullet \arrow[rd, "c_2"', no head] &                                       \\
                                   &                                   &                                    &                                     & \bullet \arrow[luuu, "c_1"', no head]
\end{tikzcd}\]
\begin{tikzpicture}[remember picture, overlay,E/.style = {rectangle, draw=red, inner xsep=0.8pt,inner ysep=0.8pt, fit=#1}]
\node[E = (m-2-3) (m-4-3)] {};
\node[E = (m-2-4) (m-4-4)] {};
\end{tikzpicture}
\vspace{9mm}
\caption{$\B c'=\B c$}
\end{subfigure}
\caption{Underlying undirected trees of the generalised tree modules over Dynkin quivers of type $\B D_n$, with $F(a)=F(a_j)=\B a$, $F(b)=F(b_j)=\B b$ and $F(c)=F(c_j)=\B c$ for $j\in\{1,2\}$}
    \label{fig:ind modules of DN}
\end{figure}
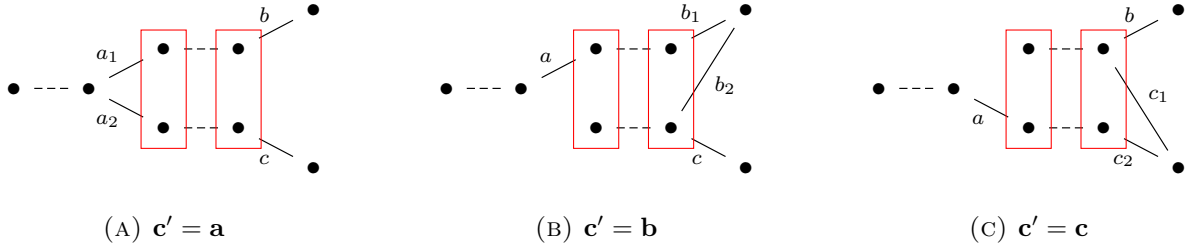

Since the arguments for the (in)decomposability of the generalised tree modules sketched in \Cref{fig:ind modules of DN} are similar, we focus our attention to a particular quiver $\B Q$ where the orientations of the arrows $\B a,\B b,\B c$ are as shown in \Cref{fig: particular DN quiver}. Here we are assuming that $\B l$ is the least such that $\dim(U_\B l)>0$.

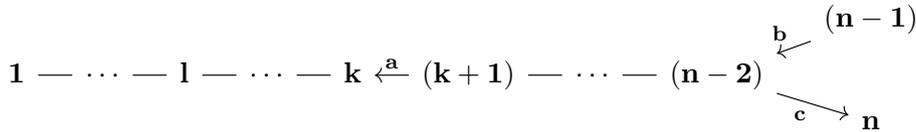
\begin{figure}[H]
    \[\begin{tikzcd}[column sep=small, row sep=2pt]
                       &                           &                         &                           &      &                                                 &                           &                               & \B{(n-1)} \arrow[ld, "\B b"'] \\
\B1 \arrow[r, no head] & \cdots \arrow[r, no head] & \B l \arrow[r, no head] & \cdots \arrow[r, no head] & \B k & \B{(k+1)} \arrow[l, "\B a"'] \arrow[r, no head] & \cdots \arrow[r, no head] & \B{(n-2)} \arrow[rd, "\B c"'] &                               \\
                       &                           &                         &                           &      &                                                 &                           &                               & \B{n}                        
\end{tikzcd}\]
    \caption{A particular Dynkin quiver $\B Q$ of type $\B D_n$}
    \label{fig: particular DN quiver}
\end{figure}

\noindent\textbf{If $T$ is as in \Cref{fig:particular module dynkin}, then $M:=F_\lambda(V_T)$ is indecomposable:} Here the bound quiver morphism $F:T\to \B Q$ is obvious from the diagram. This is consistent with \Cref{fig:ind modules of DN}(B) as $\B c'=\B b$ is the only ``correct'' choice.

\begin{figure}[H]
\vspace{10mm}
    \[\begin{tikzcd}[column sep=small, row sep=2pt,
every matrix/.append style = {name=m},remember picture, overlay]
                             &   &                                                    &                          & (n-1) \arrow[ld, "b_1"'] \arrow[lddd, "b_2"] \\
                             &   & (k+1)' \arrow[r, no head, dashed] \arrow[ld, "a"'] & (n-2)'                   &                                              \\
l \arrow[r, no head, dashed] & k &                                                    &                          &                                              \\
                             &   & (k+1)'' \arrow[r, no head, dashed]                 & (n-2)'' \arrow[rd, "c"'] &                                              \\
                             &   &                                                    &                          & n                                           
\end{tikzcd}\]
\begin{tikzpicture}[remember picture, overlay,E/.style = {rectangle, draw=red, inner xsep=0.8pt,inner ysep=0.8pt, fit=#1}]
\node[E = (m-2-3) (m-4-3)] {};
\node[E = (m-2-4) (m-4-4)] {};
\end{tikzpicture}
\vspace{9mm}
    \caption{Sketch of a tree $T$ of type described in \Cref{fig:ind modules of DN}(B)}
    \label{fig:particular module dynkin}
\end{figure}
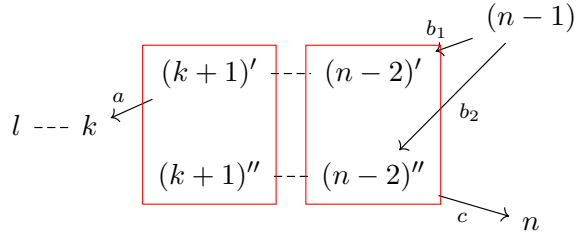

First, we will discuss some properties of the networks $\C N[1]$ and $\C N[2]$ associated with the pair $(M,M)$. The pullback network $\C N[1]$ (see \Cref{fig:N1 for the particular Dynkin module}) contains exactly two $\C R[1]$-systems, each containing exactly one triangle (highlighted in red). Therefore, the $2$-covering network $\C N[2]$ associated with the pair $(M,M)$ contains two $\C R[2]$-systems relative to $\C N[1]_0$, each containing only one hexagon. It can be verified that a traversal from $(n,m,1)$ to $(n,m,-1)$ in $\C N[2]$ for any $(n,m)\in\C N[1]_0$ contains two adjacent links of a hexagon, and therefore cannot be $\C R[2]$-free. This proves that the pair $(M,M)$ is ghost-free.

\begin{figure}[H]
    \[\begin{tikzcd}[column sep=tiny]
&                                               & {(n-1,n-1)} \arrow[lldd, "{(b_1,b_1)}"',color=red] \arrow[ldd, "{(b_2,b_1)}",color=red] \arrow[rdd, "{(b_1,b_2)}"',color=red] \arrow[rrdd, "{(b_2,b_2)}",color=red] &                                                                  &                                                \\
                                             &                                               &                                                                                                                             &                                                                  & {(n,n)}                    \\
{((n-2)',(n-2)')} \arrow[r, no head,color=red]         & {((n-2)'',(n-2)')}                            &                                                                                                                             & {((n-2)',(n-2)'')} \arrow[r, no head,color=red] \arrow[d, no head, dashed] & {((n-2)'',(n-2)'')} \arrow[d, no head, dashed]\arrow[u, "{(c,c)}"] \\
{((k+1)',(k+1)')}\arrow[d, "{(a,a)}"'] \arrow[u, no head, dashed] & {((k+1)'',(k+1)')} \arrow[u, no head, dashed] &                                                                                                                             & {((k+1)',(k+1)'')}                                               & {((k+1)'',(k+1)'')}                            \\
{(k,k)}                  &                                               &                                                                                                                             &                                                                  &                                                \\
{(l,l)} \arrow[u, no head, dashed]           &                                               &                                                                                                                             &                                                                  &                                               
\end{tikzcd}\]
    \caption{Pullback network associated with the pair $(M,M)$}
    \label{fig:N1 for the particular Dynkin module}
\end{figure}
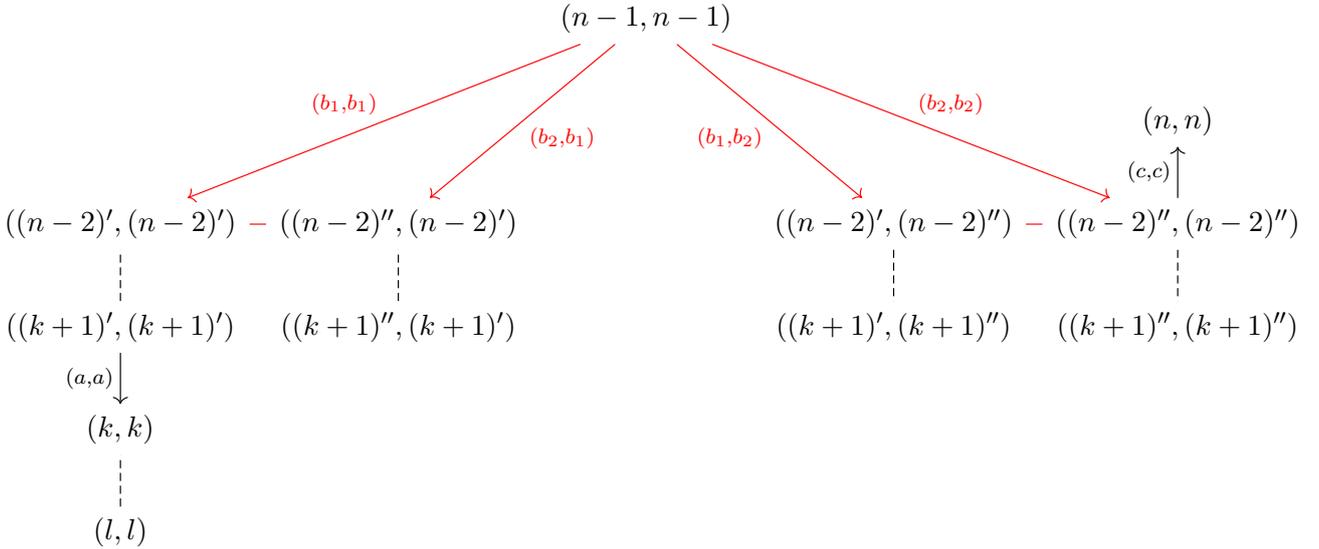

There is only one subtree of $T$ (upto relabelling) of the form $n_1\xleftarrow{a''}n\xrightarrow{b''}n_2$ or $n_1\xrightarrow{a''}n\xleftarrow{b''}n_2$ with $n_1\neq n_2$ and satisfying $F(a'')=F(b'')$, namely $(n-2)'\xleftarrow{b_1}(n-1)\xrightarrow{b_2}(n-2)''$. On the one hand, since the completeness condition for arrow $c$ is not satisfied at $((n-2)',(n-2)'',j)$ for any $j\in\{-1,1\}$, there does not exist a generalised graph map containing $((n-2)',(n-2)'',j)$. On the other hand, any generalised graph map containing $((n-2)'',(n-2)',j)$ for some $j\in\{-1,1\}$ must also contain $((k+1)'',(k+1)',j)$. However, since the completeness condition for arrow $a$ is not satisfied at $((k+1)'',(k+1)',j)$, we obtain the impossibility of the existence of a generalised graph map containing $((n-2)'',(n-2)',j)$ as well. Thus, we have shown that Condition (1) of the hypotheses of \Cref{thm: main 2} holds.

Also, note that any generalised graph map containing $(n_1,n_2)$ for $n_1\neq n_2$ in its support must contain $((n-2)',(n-2)'',j)$ or $((k+1)'',(k+1)',j)$ for some $j\in\{-1,1\}$ as its vertices. As shown in the above paragraph, the completeness condition for at least one arrow fails at these vertices. Hence, there does not exist a generalised graph map containing $(n_1,n_2)$ for $n_1\neq n_2$ in its support. This proves Condition (2) of the hypotheses of \Cref{thm: main 2}. Therefore, we conclude that $M$ is indecomposable.\\

\noindent\textbf{If $T'$ is as in \Cref{fig:particular example decomposable Dn}, then $M'=F'_\lambda(V_{T'})$ is decomposable:} Here the bound quiver morphism $F':T'\to \B Q$ is obvious from the diagram. This is an application of \Cref{thm: conv main2}; the decomposability is due to the fact that $\B c'=\B c$ is a ``wrong'' choice.

\begin{figure}[H]
\vspace{10mm}
    \[\begin{tikzcd}[column sep=small, row sep=2pt,
every matrix/.append style = {name=m},remember picture, overlay]
                             &   &                                                    &                            & (n-1) \arrow[ld, "b"'] \\
                             &   & (k+1)' \arrow[r, no head, dashed]                  & (n-2)' \arrow[rddd, "c_1"] &                        \\
l \arrow[r, no head, dashed] & k &                                                    &                            &                        \\
                             &   & (k+1)'' \arrow[r, no head, dashed] \arrow[lu, "a"] & (n-2)'' \arrow[rd, "c_2"'] &                        \\
                             &   &                                                    &                            & n                     
\end{tikzcd}\]
\begin{tikzpicture}[remember picture, overlay,E/.style = {rectangle, draw=red, inner xsep=0.8pt,inner ysep=0.8pt, fit=#1}]
\node[E = (m-2-3) (m-4-3)] {};
\node[E = (m-2-4) (m-4-4)] {};
\end{tikzpicture}
\vspace{9mm}
    \caption{Sketch of a tree $T'$ of type described in \Cref{fig:ind modules of DN}(C)}
    \label{fig:particular example decomposable Dn}
\end{figure}
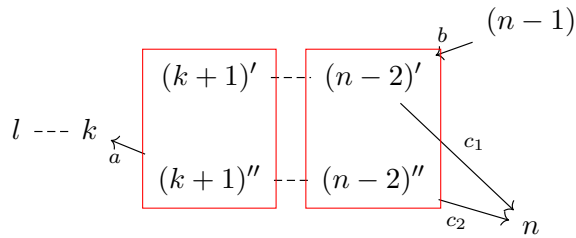

Let $T^1$ and $T^2$ be the subtrees $(k+1)'\text{- -}(n-2)'\xleftarrow{b}(n-1)$ and $l\text{- -} k\xleftarrow{a}(k+1)''\text{- -} (n-2)''$ of $T'$ respectively. The generalised graph map induced by the set $\C X\subseteq\C N[2]_0$ defined as
$$\C X:=\{(m,m,1)\mid m\in\{n-1,n\}\}\cup\{(m'',m',1)\mid k+1\leq m\leq n-2\}\cup\{(m',m',1)\mid k+1\leq m\leq n-2\}$$  satisfies the hypotheses of \Cref{thm: conv main2} since $((n-2)'',(n-2)',1)\in\C X$. As a result, we conclude the decomposability of $M$, as promised. A similar argument involving \Cref{thm: conv main2} would show that any wrong choice of $\B c'$ would lead to decomposability.

\subsection*{Acknowledgements}
The authors thank Thomas Br\"ustle for his comments on the first draft as well as for bringing the works of Ringel \cite{ringelexceptional}, Kinser \cite{Kinser}, and Katter and Mahrt \cite{katter_mahrt} to their attention.

\subsection*{Funding} Annoy Sengupta was supported by a \emph{Council of Scientific and Industrial Research (CSIR)} India - Research Grant No. 09/092(1090)/2021-EMR-I.

\subsection*{Research Support}
Annoy Sengupta reports financial support was provided by \emph{Council of Scientific and Industrial Research (CSIR)}.

\subsection*{Relationships}
There are no additional relationships to disclose.

\subsection*{Patents and Intellectual Property}
There are no patents to disclose.

\subsection*{Other activities}
There are no additional activities to disclose.
\subsection*{Data availability statement} This manuscript has no associated data.

\bibliographystyle{alpha}
\bibliography{main}
\vspace{0.2in}
\noindent{}Annoy Sengupta\\
Indian Institute of Technology Kanpur\\
Uttar Pradesh, India\\
Email 1: \texttt{annoysgp20@iitk.ac.in}\\
Email 2: \texttt{sengupta.annoy44@gmail.com}
\vspace{0.2in}

\noindent{}Corresponding Author: Amit Kuber\\
Indian Institute of Technology Kanpur\\
Uttar Pradesh, India\\
Email 1: \texttt{askuber@iitk.ac.in}\\
Email 2: \texttt{expinfinity1@gmail.com}
Phone: (+91) 512 259 6721\\
Fax: (+91) 512 259 7500
\end{document}